\documentclass[11pt]{amsart}  

\usepackage{amsmath}
\usepackage{amsthm}
\usepackage{amsfonts}
\usepackage{amssymb}
\usepackage{eucal}
\usepackage[all]{xy}
\usepackage{amsxtra}
\usepackage{appendix} 
\usepackage{tensor} 
\usepackage{url}
\usepackage{upgreek}
\usepackage{etoolbox}
\apptocmd{\sloppy}{\hbadness 10000\relax}{}{}
\usepackage{comment}

\usepackage[pdftex, breaklinks, linktocpage=true, bookmarksopen=true,bookmarksopenlevel=0,bookmarksnumbered=true]{hyperref}

\RequirePackage{color}
\definecolor{bwgreen}{rgb}{0.183,1,0.5}
\definecolor{bwmagenta}{rgb}{0.7,0.0,0.1}
\definecolor{bwblue}{rgb}{0.317,0.161,1}

\hypersetup{
pdfauthor = {\textcopyright\ Bryden Cais},
pdfcreator = {\LaTeX\ with package \flqq hyperref\frqq},
colorlinks = {true},
linkcolor={bwmagenta},	 %color red Color for normal internal links.
anchorcolor={},	 %color	black Color for anchor text.
citecolor={bwblue},	% color	green Color for bibliographical citations in text.
filecolor={},	% color cyan Color for URLs which open local files.
menucolor={},	% color	red	Color for Acrobat menu items.
runcolor={},	% color	filecolor Color for run links (launch annotations).
urlcolor={bwgreen}, %	 color	magenta	Color for linked URLs.
}

\CompileMatrices
\entrymodifiers={+!!<0pt,\fontdimen22\textfont2>}
\DeclareFontFamily{OT1}{rsfs}{}
\DeclareFontShape{OT1}{rsfs}{n}{it}{<-> rsfs10}{}
\DeclareMathAlphabet{\mathscr}{OT1}{rsfs}{n}{it}

\DeclareFontFamily{OT1}{pzc}{}
\DeclareFontShape{OT1}{pzc}{n}{it}{<->s*[2.2]pzc}{}
\DeclareMathAlphabet{\mathpzc}{OT1}{pzc}{b}{sl}

\makeatletter
\newcommand{\rmnum}[1]{\romannumeral #1}
\newcommand{\Rmnum}[1]{\expandafter\@slowromancap\romannumeral #1@}
\makeatother

\DeclareMathOperator{\Frac}{Frac}
\DeclareMathOperator{\ord}{ord}

\DeclareMathOperator{\Gal}{Gal}

\DeclareMathOperator{\Aut}{Aut}

\DeclareMathOperator{\Fr}{Fr}

\DeclareMathOperator{\sep}{sep}

\DeclareMathOperator{\im}{im}

\DeclareMathOperator{\Nm}{Nm}
\renewcommand*{\c}{\ensuremath{\mathbf{C}}}              
\newcommand*{\Z}{\ensuremath{\mathbf{Z}}}               
\newcommand*{\Q}{\ensuremath{\mathbf{Q}}}                           
\newcommand*{\N}{\ensuremath{\mathbf{N}}}

\newcommand*{\Kbar}{\overline{K}}    
          
\newcommand*{\Gm}{\ensuremath{{\mathbf{G}_m}}}

\newcommand*{\m}{\mathfrak{M}}

\newcommand*{\s}{\mathfrak{S}}

\newcommand*{\C}{\mathbf{C}}
\newcommand*{\E}{\mathscr{E}}     
\newcommand*{\F}{\mathbf{F}}

\let\danishO\O
\let\O\tempO
 
\renewcommand*{\b}{\ensuremath{\mathfrak{b}}}
\newcommand*{\mfa}{\ensuremath{\mathfrak{a}}}
\renewcommand*{\int}{\ensuremath{\mathrm{int}}}

\newcommand*{\e}{\ensuremath{\mathbf{E}}}
\renewcommand*{\a}{\ensuremath{\mathbf{A}}}

\renewcommand*{\u}[1]{\underline{#1}}
\newcommand*{\wh}[1]{\widehat{#1}}
\newcommand*{\wt}[1]{\widetilde{#1}}

\numberwithin{equation}{section}
\theoremstyle{plain}
  \newtheorem{theorem}[equation]{Theorem}
    
  \newtheorem{proposition}[equation]{Proposition}
  \newtheorem{lemma}[equation]{Lemma}
  \newtheorem{corollary}[equation]{Corollary}

\theoremstyle{definition}
  \newtheorem{definition}[equation]{Definition}
  \newtheorem{notation}[equation]{Notation}

\theoremstyle{remark}
  \newtheorem{example}[equation]{Example}
  \newtheorem{remark}[equation]{Remark}

\include{header}

\newcommand{\chris}[1]{{\color{blue} \sf $\spadesuit\spadesuit\spadesuit$ Chris: [#1]}}
\newcommand{\bryden}[1]{{\color{red} \sf $\heartsuit\heartsuit\heartsuit$ Bryden: [#1]}}

\newcommand{\todo}[1]{\vspace{5 mm}\par \noindent
\marginpar{\textsc{ToDo}}
\framebox{\begin{minipage}[c]{0.95 \textwidth}
\tt #1 \end{minipage}}\vspace{5 mm}\par}

\newcommand{\ALK}{\a_{L/K}^+}

\begin{document}
\title{Canonical Cohen rings for norm fields}

\author{Bryden Cais}
\address{University of Arizona, Tucson}
\curraddr{Department of Mathematics, 617 N. Santa Rita Ave., Tucson AZ. 85721}
\email{cais@math.arizona.edu}

\author{Christopher Davis}
\address{University of Copenhagen, Denmark}
\curraddr{Department of Mathematical Sciences,  Universitetsparken~5, 
DK-2100 K{\o}benhavn \danishO}
\email{davis@math.ku.dk}

%\thanks{
%	During the writing of this paper, the author was partially supported by an NSA Young Investigator grant
%	(H98230-12-1-0238).
%	}

%\dedicatory{}

\date{\today}

\begin{abstract}
Fix $K/\Q_p$ a finite extension and let $L/K$ be an infinite, strictly APF extension in the sense of 
Fontaine--Wintenberger \cite{Wintenberger}.  Let $X_K(L)$ denote its associated \emph{norm field}.  The goal of this paper is to associate to $L/K$, in a canonical and functorial way, a $p$-adically complete subring $\a_{L/K}^+ \subset \widetilde{\a}^+$ whose reduction modulo~$p$ is contained in the valuation ring of $X_K(L)$.  When the extension $L/K$ is of a special form, which we call a \emph{$\varphi$-iterate extension}, we prove that $X_K(L)$ is (at worst) a finite purely inseparable extension of $\Frac (\a_{L/K}^+/p\a_{L/K}^+)$.  The class of $\varphi$-iterate extensions includes all Lubin--Tate extensions, as well as many other
extensions such as the non-Galois ``Kummer" extension occurring in work of Faltings, Breuil, and Kisin.
In particular, our work provides a canonical and functorial construction of
every characterietic zero lift of the norm fields that have thus far played a foundational role in (integral) $p$-adic Hodge theory, as well as many other cases which have yet to be studied.
%the particular cases that $L/K$ is a cyclotomic extension (as studied by Fontaine), a Lubin--Tate extension (as studied by Kisin--Ren and others), 
%as well as many other cases.  \chris{Are you happy with the previous sentence?} Our constructions rely heavily on Witt vectors evaluated at characteristic~0 rings, and on a generalization of Witt vectors studied by Jim Borger.
\end{abstract}

\maketitle

\section{Introduction}

Let $K$ be a finite extension of $\Q_p$ with residue field $k_K$, fix an algebraic closure
$\Kbar$ of $K$, and let $G_K:=\Gal(\Kbar/K)$ be the absolute Galois group of $K$.
A natural approach to the classification of $p$-adic representations of $G_K$
is to first study their restrictions 
to subgroups $G_L$ of $G_K$ for $L/K$ an infinite extension.  When $L/K$ is {\em arithmetically profinite}
(APF) in the sense that the higher ramification subgroups $G_K^v G_L$ are {\em open} in $G_K$,
a miracle occurs: by the field of norms machinery of Fontaine--Wintenberger \cite{Wintenberger}, 
there is a canonical isomorphism of Galois groups 
\begin{equation}
	\Gal(\Kbar/L)\simeq \Gal(X_K(L)^{\sep}/X_K(L))\label{CatEquiv},
\end{equation} 
where $X_K(L)$ is the {\em field of norms} of $L/K$, which is a discretely valued field of equicharacteristic 
$p$, noncanonically isomorphic to $k_L(\!(u)\!)$, where $k_L$ is the residue field of $L$. 
%Fix a finite extension $K$ of $\Q_p$ with residue field $k$, and let $L/K$ an infinite {\em arithmetically profinite}
%(APF) extension in the sense of \cite{Wintenberger}.  
%%The norm fields
%%machinery of Fontaine--Wintenberger \cite{Wintenberger} allows one to canonically associate 
%%to $L/K$ a functor $X_{L/K}(\cdot)$ on the category of finite (separable) extensions
%%of $L$, valued in the category of finite, separable extensions
%%of the norm field $X_K(L):=X_{L/K}(L)$
%Attached to $L/K$ is its {\em field of norms} $X_K(L)$, which is an equicharacteristic $p$ discretely valued field,
%noncanonically isomorphic to $k(\!(x)\!)$, that is naturally equipped with a continuous action of $\Aut(L/K)$.
%The norm fields machinery of Fontaine--Wintenberger \cite{Wintenberger}
%extends the construction of $X_K(L)$ to a functor $X_{L/K}(\cdot)$ from the category of finite (separable) extensions
%of $L$ to the category of finite, separable extensions of $X_K(L)=X_{L/K}(L)$, and the main theorem
%of \cite{Wintenberger} is that this functor is an equivalence of categories.  In particular,
%one has a natural isomorphism of topological groups 
As a consequence of (\ref{CatEquiv}), Fontaine proved \cite{Fontaine90}
that $G_L$-representations on $\F_p$-vector spaces are classified by \'etale $\varphi$-modules over $X_K(L)$.
In order to lift this to a classification of $G_L$-representations on $\Z_p$-modules, it is necessary to lift the norm field $X_K(L)$ together with its Frobenius endomorphism to characteristic zero.  Provided that this can be done,
one then aims to recover the entire $G_K$-representation via appropriate descent data from $G_L$ to $G_K$.
This strategy has successfully been carried out in three important cases:
\begin{enumerate}
	\item For $L=K(\mu_{p^{\infty}})$, Fontaine's work \cite{Fontaine90} provides
	the necessary lift of $X_K(L)$, and results in the theory 
	of $(\varphi,\Gamma)$-modules which has played a pivotal role in the development
	and major applications of $p$-adic Hodge theory.
	This lift of the norm field is also essential in Berger's theory of Wach modules,
	which provides a classification of (certain) crystalline representations
	of $G_K$ when $K$ is an abelian extension of an absolutely unramified
	extension of $\Q_p$ \cite{Berger}, \cite{BergerBreuil}.\label{phiG}  
	
	\item More recently, the natural generalization
	of Fontaine's theory in which the cyclotomic extension of $K$ is replaced by an arbitrary Lubin--Tate extension
	determined by a Lubin--Tate formal group over a subfield $F$ of $K$ has been studied
	by several authors \cite{Fourquaux}, \cite{Chiarellotto}, \cite{KisinRen}.  As in (\ref{phiG})
	(which is the special case $F=W(k_K)$ with formal group $\wh{\Gm}$), the underlying
	formal group plays an essential role in constructing a characteristic zero 
	lift of the norm field $X_K(L)$ that has enough structure to encode the desired descent 
	data from $G_L$ to $G_K$.
	
	\item For $L=K(\pi_0^{1/p^{\infty}})$ the (non-Galois) ``Kummer" extension obtained by 
	adjoining a compatible system of $p$-power roots of a uniformizer $\pi_0$ in $K$, 
	the resulting theory of $\varphi$-modules (due to Fontaine) 
	has been extensively exploited by Faltings \cite{Faltings}, Breuil \cite{BreuilNormes} \cite{BreuilIntegral}, and Kisin \cite{KisinCrystal} 
	in their development of integral $p$-adic Hodge theory.
	In this case, the ramification structure of $L/K$ is such that for {\em crystalline} $G_K$-representations,
	there is a {\em unique} descent datum from $G_L$ to $G_K$ so that one obtains a classification of
	these representations by $\varphi$-modules alone (without any $\Gamma!$).\label{KisinExt}
\end{enumerate}

In each of these special cases, the required lift of the norm field to characteristic zero is constructed
in a rather {\em ad hoc} manner, and it is not in general clear that it is independent of the choices
made to construct it.

In this paper, we propose a candidate for a {\em canonical} and {\em functorial} lift to characteristic zero 
of the valuation ring $A_K(L)$ of the norm field $X_K(L)$.  In order to encompass the aforementioned special cases, we work in
slightly greater generality: let $F$ be a subfield of $K$
with residue field $k_F\subseteq k_K$ of cardinality $q:=|k_F|$, choose a uniformizer $\pi$ of $F$,
and for any $W(k_F)$-algebra $A$, write $A_F:=A\otimes_{W(k_F)} \O_F$.
%Let $L/K$ be an arbitrary {\em strictly} APF extension; 

\begin{theorem}\label{MT}
	Let $L/K$ be an arbitrary strictly\footnote{The strictness hypothesis guarantees
that $A_K(L)$ is naturally a subring of Fontaine's ring $\wt{\e}^+:=\varprojlim_{x\mapsto x^p} \O_{\Kbar}/(p)$.} APF extension and let $k_L$ denote the residue field of $L$.  There is a canonical $W(k_L)_F$-subalgebra
	$\a_{L/K}^+$ of $\wt{\a}^+_F:=W(\wt{\e}^+)_F$, depending only on 
	$L/K$ and the choices $F$, $\pi$, such that:
	\begin{enumerate}
	 	\item $\a_{L/K}^+$ is $\pi$-adically complete and separated, and
		closed for the weak topology on $\wt{\a}^+_F$.\label{MT1}
		
		\item $\a_{L/K}^+$ is stable under the $q$-power Frobenius automorphism of $\wt{\a}_F^+$.  	 
		\item Via the inclusion $\a_{L/K}^+\hookrightarrow \wt{\a}_F^+$, the residue ring
		$\a_{L/K}^+/\pi \a_{L/K}^+$ of $\a_{L/K}^+$ is naturally
		a $k_L$-subalgebra of $A_K(L)$.\label{MT3}
		
		\item The ring $\a_{L/K}^+$ is functorial in $L/K$: for any strictly APF extension $L'/K$
		and any $K$-embedding $\tau:L\hookrightarrow L'$ with $L'/\tau(L)$ finite and of wild ramification degree
		a power\footnote{This restriction is vacuous if $q=p$, {\em i.e.} if $F/\Q_p$
		is totally ramified} of $q=|k_F|$, there is associated a canonical embedding 
		$\a_{L/K}^+\hookrightarrow \a_{L'/K}^+$
		whose reduction modulo $\pi$ is the restriction of the map on norm fields induced by $\tau$ \cite[3.1.1]	
		{Wintenberger}; in particular, $\a_{L/K}^+$ is naturally equipped with an action of $\Aut(L/K)$.\label{MT4}
	\end{enumerate}
\end{theorem}
	
For general strictly APF extensions $L/K$, the canonical
homomorphism  $W(k_L)_F\rightarrow \a_{L/K}^+$ is often an {\em isomorphism},
and our candidate ring $\ALK$ is too small;
following recent work of Berger \cite{BergerLift},
we prove that this happens, for example, whenever $L/K$ is Galois with group that is 
a $p$-adic Lie group admitting no abelian quotient by a finite subgroup 
(see Remark \ref{pessimism}).
%be too small 
%as we have been unable to prove that the norm field $X_K(L)$
%is of finite degree over the fraction field of the residue ring of $\a_{L/K}^+$.
%In order to prove that this residue ring $\a_{L/K}^+/\pi\a_{L/K}^+$ is large enough, a natural strategy is to try to
%lift a uniformizer of $A_K(L)$ to $\a_{L/K}^+$.  If such a lift $u\in \a_{L/K}^+$
%exists, then from Theorem \ref{MT} (\ref{MT1}) one deduces
%a (non-canonical) isomorphism $\a_{L/K}^+\simeq W(k_L)_F[\![u]\!]$, so that 
%$\a_{L/K}^+$ is a lift of {\em finite height} in the sense of Berger \cite[\S4]{BergerLift}.
%However, according to the main result of \cite{BergerLift}, if $L/K$ is {\em Galois},
%and a finite-height lift of $A_K(L)$ with a natural $\Gal(L/K)$ action
%(as in Theorem \ref{MT} (\ref{MT4})) exists, then $\Gal(L/K)$ is necessarily {\em abelian}.
Nevertheless, we isolate a large class of (not necessarily Galois) strictly APF extensions $L/K$---containing each of the extensions (\ref{phiG})--(\ref{KisinExt}) above
as special cases---for which we are able to prove that $\a_{L/K}^+$
provides the desired functorial lift of the valuation ring in the norm field of $L/K$:

\begin{definition} \label{iterated def}
	With notation above, let $\pi_0$ be any choice of uniformizer of $K$, and
	let $\varphi(x) \in \O_F[\![x]\!]$ be any power series satisfying 
	$\varphi(0) = 0$ and $\varphi(x) \equiv x^q \bmod \pi \O_F$.  
	Beginning with $\pi_0 \in \O_K$, recursively choose $\pi_{i+1} \in \O_{\Kbar}$ 
	 a root of $\varphi(x) - \pi_i$, and let $L := \cup_i K(\pi_i)$.  
	We call a strictly APF extension $L/K$ of this form\footnote{Together with Jonathan Lubin,
	the authors expect to show that the strictly APF condition is automatically satisfied 
	by every extension of this type.
	} a \emph{$\varphi$-iterate extension}.  
\end{definition}

For any $\varphi$-iterate extension $L/K$, we prove that our candidate ring $\a_{L/K}^+$
lifts the valuation ring $A_{K}(L)$ of the norm field of $L/K$, up to possibly 
a finite, purely inseparable extension (which for applications to $p$-adic Hodge theory
is sufficient; see {\em e.g.} \cite[3.4.4 (a)]{Fontaine90}):

\begin{theorem}\label{MT-iterate}
	Let $L/K$ be a $\varphi$-iterate extension.  There exists $u\in \a_{L/K}^+$
	whose reduction modulo $\pi$ is the $q^d$-th power of a uniformizer in $A_K(L)$,
	for some integer $d\ge 0$.  In particular, $\a_{L/K}^+\simeq W(k_L)_F[\![u]\!]$,
	and the norm field $X_K(L)$ is a finite, purely inseparable extension of 
	$\Frac(\a_{L/K}^+/\pi\a_{L/K}^+)$.
\end{theorem}
 
Our work thus provides a {\em unified} and {\em canonical} construction of every characteristic zero lift of the norm fields
that have thus far played a foundational role in (integral) $p$-adic Hodge theory, and raises the possibility
of a general theory for arbitrary strictly APF extensions. 

Throughout this article, we keep the following conventions and notations:
\begin{notation} \label{K notation}
We fix a finite extension $F$ of $\Q_p$ and a choice $\pi$ of uniformizer of $\O_F$, and we denote
by $\mathfrak{m}_F$ and $k_F$ the maximal ideal and residue field of $\O_F$, respectively.
We write $p^a=q:=|k_F|$, and for any $W(k_F)$-algebra $A$ we set $A_F:=A\otimes_{W(k_F)} \O_F$.
%, and normalize the valuation $v$ of $F$ so that $v(F) = \Z$.  
When we consider a strictly APF extension $L/K$, the bottom field $K$ is assumed to contain $F$.  We let $K_0$ and $K_1$ denote, respectively, the maximal unramified and tamely ramified subextensions of $L/K$.  We let $k$ denote the residue field of $L$ (or equivalently, of $K_0$).  
Our construction of $\a_{L/K}^+$ will depend implicitly on our choices $F$ and $\pi$, though we suppress
this in our notation.  
%We will also let $\pi_0$ denote a fixed uniformizer of $K$ and $k$ denote the residue field of $K$, where $K/F$ is some fixed finite extension which should be clear from context.
\end{notation}

\subsection*{Acknowledgments} The authors are very grateful to Laurent Berger, Jim Borger, Lars Hesselholt, Kiran Kedlaya,
Abhinav Kumar, Ruochuan Liu, Jonathan Lubin, and Anders Thorup for many helpful discussions.   
The first author is supported by an NSA ``Young Investigator" grant (H98230-12-1-0238).
The second author is partially supported by the Danish National Research Foundation through the Centre for Symmetry and Deformation (DNRF92).

\section{Generalized Witt vectors}

Let $F, q, \pi, \varphi(x)$ be defined as in Notation~\ref{K notation} and Definition~\ref{iterated def}.  Our construction of $\a_{L/K}^+$ is inspired by the alternative description 
$\wt{\a}^+ \cong \varprojlim_{\Fr} W(\O_{\c_F})$ due to the second author and Kedlaya \cite{DavisKedlaya} (see also Proposition~\ref{prop: Fontaine} below),
in which the inverse limit is taken along the Witt vector Frobenius map.  To account for ramification and to account for $q$-powers instead of $p$-powers, we use generalized Witt vectors.

The power series $\varphi$ determines a unique continuous $\O_F$-algebra homomorphism
\[
\O_F[\![ x]\!] \rightarrow \O_F[\![ x]\!], \qquad x \mapsto \varphi(x).
\]
In the special case that $F = \Q_p$, meaning $q = p$ and we can take $\pi = p$, then such a ring homomorphism determines a unique ring homomorphism
\[
\lambda_{\varphi}: \O_F[\![ x]\!] \rightarrow W(\O_F[\![ x]\!])
\]
which is a section to the projection map $W(\O_F[\![ x]\!]) \rightarrow \O_F[\![ x]\!]$ and which satisfies
\[
\lambda_{\varphi} \circ \varphi = \Fr \circ \lambda_{\varphi},
\]
where $\Fr$ is the Witt vector Frobenius.  In fact, a completely analogous result holds for the general fields $F/\Q_p$ considered in this paper; see Proposition~\ref{lambda prop}.  For this result, we will require a generalization of the classical $p$-typical Witt vectors.  These generalized Witt vectors have appeared in the work of Drinfeld \cite[\S1]{Dri76} and Hazewinkel \cite[(18.6.13)]{Haz78}.  Our exposition follows that of Borger\footnote{Borger has indicated to the authors that 
the constructions and arguments of his paper work equally well under the more general hypotheses
$\varphi(x)\equiv x^{q^j}\bmod \pi$ for any $j\ge 1$ and $\varphi(0)=0$; we stick to the present assumptions for simplicity.} in \cite{BorgerGeometry}.  The goal of this section is to recall the properties of these generalized Witt vectors.  

Define functors $\bullet^{\N}$ from sets to sets, from $\O_F$-algebras to sets, and from $\O_F$-algebras to $\O_F$-algebras in the obvious way.  

\begin{definition} \label{Witt variant} 
Let $W_{\pi}( \bullet )$ denote the unique functor from $\O_F$-algebras to $\O_F$-algebras satisfying the following two properties:
\begin{itemize}
\item The induced functor from $\O_F$-algebras to sets agrees with $\bullet^{\N}$.  
\item The ghost map $W_{\pi}(\bullet) \rightarrow \bullet^{\N}$ given by
\[
(a_0, a_1, a_2, \ldots) \mapsto (a_0, a_0^q + \pi a_1, a_0^{q^2} + \pi a_1^q + \pi^2 a_2, \ldots)
\] 
is a natural transformation of functors.
\end{itemize}
\end{definition}

\begin{remark}
Perhaps the notation $W_{\pi, q}$ would be more precise.  We hope the notation $W_{\pi}$ is sufficient to remind the reader that these are not the usual $p$-typical Witt vectors.
\end{remark}

This defines the ring operations on $W_{\pi}(R)$ for any $\O_F$-algebra $R$ in the usual way:  If $R$ is $\pi$-torsion free, then the ghost map is injective, and so the ring operations are uniquely determined by the requirement that the ghost map be a ring homomorphism.  In the general case, we take a surjection from a $\pi$-torsion free ring $S$ to $R$; we then use the ring operations on $S$ to define the ring operations on $R$.  

\begin{proposition}
The functor $W_{\pi}(\bullet)$ defined above exists.  In particular, for any $\O_F$-algebra $R$, $W_{\pi}(R)$ is an $\O_F$-algebra.  
\end{proposition}

\begin{proof}
The functor $W_{\pi}(\bullet)$ is constructed in a more abstract way in \cite{BorgerGeometry}.  For the agreement with our definition, see \cite[\S{3.1}]{BorgerGeometry}.  See also page~235 of {\em loc.~cit.} for a more direct formulation.
\end{proof}

\begin{lemma} \label{p-powers and valuations}
Let $R$ denote an $\O_F$-algebra, and let $a,b \in R$ satisfy $a \equiv b \bmod \pi^i$.  Then $a^q \equiv b^q \bmod \pi^{i+1}$.  In particular, if $a \equiv b \bmod \pi$, then $a^{q^j} \equiv b^{q^j} \bmod \pi^{j+1}$.  
\end{lemma}

\begin{proof}
The second statement follows from the first by induction.  To prove the first statement, we write $b = a + \pi^i x$ and compute
\[
b^q = (a + \pi^i x)^q \equiv a^q + \pi^{qi} x^q \bmod p\pi^i
\]
by the binomial theorem.  Using that $p \in (\pi)$, this completes the proof.
\end{proof}

\begin{proposition}[Generalized Dwork Lemma] \label{dwork}
Let $R$ denote a $\pi$-torsion free $\O_F$-algebra, and let $\phi$ denote an $\O_F$-algebra homomorphism such that
\[
\phi(x) \equiv x^{q} \bmod \pi.
\]  
Let $(y_{j})$ denote a sequence of elements such that
\[
\phi\left(y_{i-1}\right) \equiv y_{i} \bmod \pi^i.
\]
Then the sequence $(y_{j})$ is in the image of the ghost map.  
\end{proposition}

\begin{proof}
We will construct $(x_{j})$, the pre-image of $(y_{j})$, inductively.  Constructing $x_0$ is trivial.  Now assume we have constructed $(x_0, \ldots, x_{n-1})$.  Then we wish to choose $x_{n}$ such that
\[
\sum_{j = 0}^{n} \pi^j x_{j}^{q^{n-j}} = y_{n}.
\]
We know by our inductive hypothesis that 
\[
\sum_{j = 0}^{n-1} \pi^j x_{j}^{q^{n-1 - j}} = y_{n-1}.  
\]
Applying $\phi$ to both sides
\[
\sum_{j = 0}^{n-1} \pi^j \phi\left(x_{j}\right)^{q^{{n-1} - j}} = \phi(y_{n-1}).
\]
The right side is congruent to $y_{n}$ modulo $\pi^n$.  On the other hand, $\phi(x_j) \equiv x_j^q \bmod \pi$, and so by Lemma~\ref{p-powers and valuations}, $\phi(x_j)^{q^{n-1-j}} \equiv x_j^{q^{n-j}} \bmod \pi^{n-j}$.  This shows 
\[
\sum_{j = 0}^{n-1} \pi^j \phi\left(x_{j}\right)^{q^{{n-1} - j}} \equiv \sum_{j = 0}^{n-1} \pi^j x_{j}^{q^{n - j}} \bmod \pi^n. 
\]
Hence we can find $x_n$ such that 
\[
\sum_{j = 0}^{n} \pi^j x_{j}^{q^{n - j}} = y_n,
\]
which completes the construction.
\end{proof}

\begin{definition}[{\cite[\S{1.7}]{BorgerGeometry}}] \label{Frobenius definition}
The {\em generalized Witt vector Frobenius} is the unique natural transformation $\Fr: W_{\pi}(\bullet) \rightarrow W_{\pi}(\bullet)$ which has the following effect on ghost components:
\[
(a_0, a_1, a_2, \ldots) \mapsto (a_1, a_2, \ldots).
\]
In particular, for any fixed $\O_F$-algebra $A$, the map $\Fr: W_{\pi}(A) \rightarrow W_{\pi}(A)$ is an $\O_F$-algebra homomorphism.  
\end{definition}

%\begin{remark}
%We hope our two uses of $F$, for our base field $F/\Q_p$ and for the Witt vector Frobenius, will not cause confusion.
%\end{remark}

\begin{definition}[{\cite[\S{3.7}]{BorgerGeometry}}]
The {\em generalized Witt vector Verschiebung} $V$ is defined on Witt vector components by
\[
(a_0, a_1, a_2, \ldots) \mapsto (0, a_0, a_1, \ldots).
\]
It is $\O_F$-linear but not multiplicative.
\end{definition}

\begin{lemma} \label{Frobenius polynomials}
The Witt vector Frobenius has the following description in terms of Witt components:
\[
(a_0, a_1, \ldots) \mapsto (a_0^q + \pi f_0(a_0, a_1), a_1^q + \pi f_1(a_0, a_1, a_2), \ldots).
\]
In other words, for arbitrary $i$, the $i$-th Witt component of $\Fr(\underline{a})$ has the form $a_i^{q} + \pi f_i$, where $f_i$ is some universal polynomial over $\O_K$ in the variables $a_0, \ldots, a_{i+1}$.  Moreover, $f_i$ is homogeneous of degree $q^{i+1}$ under the weighting in which $a_{j}$ has weight $q^j$.
\end{lemma}

\begin{proof}
The proof of the analogous result in the $p$-typical case \cite[Lemma~1.4]{DavisKedlayaFrobenius} easily adapts to this case.  The idea for proving homogeneity is to use induction and the fact that the ghost components have this property.  See also \cite[Lemma~3.2]{BorgerGeometry} for a proof of everything except homogeneity.
\end{proof}

\begin{lemma} \label{FV lemma}
If $R$ is any $\O_F$-algebra, we have $\Fr V = \pi$, where $\pi$ on the right side refers to the map $W_{\pi}(R) \rightarrow W_{\pi}(R)$ given by multiplication by $\pi$.  If $R$ is a $k_F = \O_F/\pi$-algebra, then we furthermore have $V\Fr = \Fr V = \pi$.
\end{lemma}

\begin{proof}
The fact that $\Fr V = \pi$ is clear in terms of ghost components.  Using Lemma~\ref{Frobenius polynomials}, it is clear in terms of Witt components that when $\pi = 0$ in $R$, then $\Fr V = V\Fr$.
\end{proof}

\begin{corollary} \label{mult by pi char p}
  If the $\O_F$-algebra $R$ is a $k_F = \O_F/\pi$-algebra, then multiplication by $\pi$ on $W_{\pi}(R)$ has the following effect on Witt coordinates:
\[
\pi: (a_0, a_1, a_2, \ldots) \mapsto (0, a_0^q, a_1^q, \ldots).
\]
\end{corollary}

\begin{proposition} \label{perfect k-algebra}
If the $\O_F$-algebra $R$ is a perfect $k_F$-algebra, then $W_{\pi}(R)$ is the unique $\pi$-adically complete and separated $\pi$-torsion free $\O_F$-algebra which satisfies $W_{\pi}(R)/\pi W_{\pi}(R) \cong R$.
\end{proposition}

\begin{proof}
Corollary~\ref{mult by pi char p} makes it clear that when $R$ is a $k_F$-algebra for which the $q$-power map is injective, then $W_{\pi}(R)$ is $\pi$-torsion free, $\pi$-adically complete, and $\pi$-adically separated.  This corollary, together with surjectivity of the $q$-power map, also makes it clear that $\pi W_{\pi}(R) = V(W_{\pi}(R))$, and so the residue ring is isomorphic to $R$.  

It remains to prove that $W_{\pi}(R)$ is the unique such ring.  The proof is similar to the usual proof for Witt vectors; see for example \cite[\S{II.5}]{Serre79}.  The idea is to first note that if $a \equiv b \bmod \pi$, then $a^{q^i} \equiv b^{q^i} \bmod \pi^i$ as in Lemma~\ref{p-powers and valuations}.  Using this fact one constructs a unique family of Teichm\"uller lifts, and then notes that elements in a ring satisfying the properties of the proposition can be expressed uniquely as $\pi$-adic combinations of Teichm\"uller lifts.
\end{proof}

Let $\wt{\e}^+ := \varprojlim_{x \mapsto x^p} \O_{\c_F}/p\O_{\c_F}$; this is a perfect ring of characteristic~$p$ and its fraction field $\wt{\e}$ is algebraically closed.  By \cite[\S4]{Wintenberger}, the norm fields which we seek to lift can be embedded naturally into $\wt{\e}$.  Note that the usual $p$-typical Witt vectors appear on the right side of the following corollary.

\begin{corollary} \label{e and a corollary}
$W_{\pi}(\wt{\e}^+) \cong \wt{\a}^+_F := W(\wt{\e}^+) \otimes_{W(k_F)} \O_F$.
\end{corollary}

\begin{proof}
This follows from the uniqueness assertion of Proposition~\ref{perfect k-algebra}. 
\end{proof}

Borger's approach to Witt vectors emphasizes the importance of the following result.  It is also of central importance to our work, in that we use it to lift a (power of a) uniformizer for $\e_{L/K}^+$ to our Cohen ring $\ALK$ in the case that $L/K$ is a $\varphi$-iterate extension.  This existence of a lift is the only part of our construction which fails to work for arbitrary strictly APF extensions $L/K$.

\begin{proposition} \label{lambda prop}
Let $R$ denote a $\pi$-torsion free $\O_F$-algebra, and let $\varphi: R \rightarrow R$ denote an $\O_F$-algebra homomorphism that induces the $q$-power Frobenius modulo $\pi$.  Then there is a unique $\O_F$-algebra homomorphism 
\[
\lambda_{\varphi}: R \rightarrow W_{\pi}(R)
\]
which is a section to the projection $W_{\pi}(R) \rightarrow R$, and 
such that
\[
\lambda_{\varphi} \circ \varphi = \Fr \circ \lambda_{\varphi}.
\]
\end{proposition}

\begin{proof}
See \cite[Proposition~1.9(c)]{BorgerGeometry}.  The $\pi$-torsion free requirement in our statement of the proposition corresponds to the ``$E$-flatness'' condition in Borger's paper.
\end{proof}

\begin{corollary} \label{ghost comps lambda}
Keep notation as in Proposition~\ref{lambda prop}.  For any $r \in R$, the ghost components of $\lambda_{\varphi}(r)$ are $(r, \varphi(r), \varphi^2(r), \ldots)$.  
\end{corollary}

\begin{proof}
Notice that for any $n \geq 1$, 
\[
\lambda_{\varphi} \circ \varphi^n(r) = \Fr^n \circ \lambda_{\varphi}(r).
\]  
Consider the first Witt component (which is also the first ghost component) of both sides.  Consider the left-hand side first.  Because $\lambda_{\varphi}$ is a section to the first Witt component, we have that this first Witt component is $\varphi^n(r)$.   Now consider the right-hand side.  By the definition of the Witt vector Frobenius (Definition~\ref{Frobenius definition}), the first ghost component of $\Fr^n \circ \lambda_{\varphi}(r)$ is the $(n+1)$-st ghost component of $\lambda_{\varphi}(r)$.  This completes the proof. 
\end{proof}

The preceding results, especially Proposition~\ref{lambda prop}, are the fundamental properties we will require of $W_{\pi}(R)$.  We conclude this section with several lemmas of a more technical nature which we will require in our proofs.  

\begin{lemma} \label{multiplication by pi}
If $\underline{x} = (x_0, x_1, x_2, \ldots) \in W_{\pi}(R)$ is divisible by $\pi^i$, then $x_j \in \pi^{i-j}R$ for all $0 \leq j \leq i$.
\end{lemma}

\begin{proof}
This follows easily using induction and ghost components.
\end{proof}

\begin{lemma} \label{W of pi-torsion free}
If the $\O_F$-algebra $R$ is $\pi$-torsion free, then $W_{\pi}(R)$ is also $\pi$-torsion free.
\end{lemma}

\begin{proof}
If $\pi \underline{x} = 0$ for $\underline{x} \in W_{\pi}(R)$, then at least the component $x_0$ is zero (because $R$ is assumed $\pi$-torsion free).  But using the fact that Verschiebung is injective and $\O_F$-linear, we reduce to this case.
\end{proof}

\begin{lemma} \label{W separated}
If the $\O_F$-algebra $R$ is $\pi$-adically separated, then $W_{\pi}(R)$ is $\pi$-adically separated.% and \chris{this is new} is complete for the component-wise $\pi$-adic topology.
\end{lemma}

\begin{proof}
Assume $\underline{x} \in W_{\pi}(R)$ is non-zero, and say for example that $x_j \neq 0$.  Then by the assumption that $R$ is $\pi$-adically separated, we know that $x_j \not\in \pi^N R$ for some suitably large $N$.  But then by Lemma~\ref{multiplication by pi}, we know that $\underline{x} \not\in \pi^{j + N} W_{\pi}(R)$.  
%To show that $W_{\pi}(R)$ is complete for the component-wise $\pi$-adic topology, assume $\underline{x}^{(0)}, \underline{x}^{(1)}, \ldots \in W_{\pi}(R)$ is a sequence of Witt vectors, the components of which form $\pi$-adic Cauchy sequences.  Again using Lemma~\ref{multiplication by pi}, we see that the individual components are $\pi$-adic Cauchy sequences in $R$, and thus have well-defined limits.  \chris{Unfinished.  Check that if suitably many leading Witt components are suitably divisible by $\pi$, then the Witt vector itself is divisible by $\pi$.}
\end{proof}

\begin{remark}
The authors suspect that when $R$ is $\pi$-adically complete and $\pi$-adically separated, then $W_{\pi}(R)$ is also $\pi$-adically complete.  It is not hard to find a candidate limit by working componentwise, but proving it is actually the limit seems difficult, because $\pi$-divisibility of elements in $W_{\pi}(R)$ seems difficult to detect when $R$ is $\pi$-torsion free.
\end{remark}

\begin{lemma} \label{converging ghost components}
Let $(R, \mathfrak{m})$ denote a complete local $\O_F$-algebra such that $\pi \in \mathfrak{m}$.  Let $\underline{x}_0, \underline{x}_1, \ldots \in W_{\pi}(R)$ denote a sequence of Witt vectors such that the ghost components
\[
(x_{i0}, x_{i1}, \ldots) \in R^{\N}
\] 
converge termwise to some sequence $(x_{\infty 0}, x_{\infty 1}, \ldots)$.  Then there is a Witt vector $\underline{x}_{\infty} \in W(R)$ whose ghost components are equal to $(x_{\infty 0}, x_{\infty 1}, \ldots)$.  In fact, the Witt vector components of $\underline{x}_0, \underline{x}_1, \ldots$ converge termwise to $\underline{x}_{\infty}$.
\end{lemma}

\begin{proof}
We work momentarily in $R_{\pi} := R\left[ \frac{1}{\pi} \right]$. In this setting, the ghost map $W_{\pi}(R_{\pi}) \rightarrow R_{\pi}^{\N}$ is an isomorphism, and the inverse is given in terms of sequences of polynomials.  In particular, the components of the inverse map are $\mathfrak{m}$-adically continuous functions.   The element $\underline{x}_{\infty}$ exists in $W_{\pi}(R_{\pi})$, and its components are $\mathfrak{m}$-adic limits of elements in $R$.  Hence its components are themselves elements of $R$, and the claim follows.
\end{proof}

\section{Witt vector constructions} \label{Witt constructions section}

As in Notation~\ref{K notation}, let $F/\Q_p$ be a finite extension, $q := |k_F|$ its residue cardinality, and $\pi$ a uniformizer of $\O_F$.  Let $W_{\pi}(R)$ denote generalized Witt vectors with coefficients in $R$, as in Definition~\ref{Witt variant}.  Consider the inverse system $\varprojlim_{\Fr} W_{\pi}(\O_{\c_F})$, where the transition maps are the Witt vector Frobenius as in Definition~\ref{Frobenius definition}.  We first note that this Witt vector inverse limit enjoys the same functoriality properties as the Witt vectors.

\begin{remark}
Let $\gamma: \O_{\c_F} \rightarrow \O_{\c_F}$ denote any $\O_F$-algebra homomorphism.  This extends to a ring homomorphism $W_{\pi}(\gamma): W_{\pi}(\O_{\c_F}) \rightarrow W_{\pi}(\O_{\c_F})$ by Witt vector functoriality, and because the components of the Witt vector Frobenius are defined by polynomials with coefficients in $\O_F$ (Lemma~\ref{Frobenius polynomials}), we have that $\Fr \circ W_{\pi}(\gamma) = W_{\pi}(\gamma) \circ \Fr$.  Hence $\gamma$ induces a ring homomorphism 
\[
\varprojlim_{\Fr} W_{\pi}(\O_{\c_F}) \rightarrow \varprojlim_{\Fr} W_{\pi}(\O_{\c_F}), \qquad (\underline{x}_1, \underline{x}_q, \ldots) \mapsto (W_{\pi}(\gamma)(\underline{x}_1), W_{\pi}(\gamma)(\underline{x}_q), \ldots).
\]
\end{remark}

There is also a natural action of $\Fr$ on $\varprojlim_{\Fr} W_{\pi}(\O_{\c_F}) $, defined as follows.

\begin{definition} \label{def: Frob}
Define the Frobenius ring homomorphism on $\varprojlim_{\Fr} W_{\pi}(\O_{\c_F}) $ by the formula
\[
\Fr:\varprojlim_{\Fr} W_{\pi}(\O_{\c_F})  \rightarrow \varprojlim_{\Fr} W_{\pi}(\O_{\c_F}) , \qquad (\underline{x}_1, \underline{x}_q, \underline{x}_{q^2}, \ldots) \mapsto (\Fr(\underline{x}_1), \underline{x}_1, \underline{x}_q, \ldots).
\]
This definition is equivalent to applying $\Fr$ to each entry in the inverse system.
\end{definition}

Despite the fact that the Witt vector Frobenius $\Fr: W_{\pi}(\O_{\c_F}) \rightarrow W_{\pi}(\O_{\c_F})$ is neither injective nor surjective, the above map $\Fr: \varprojlim_{\Fr} W_{\pi}(\O_{\c_F})  \rightarrow \varprojlim_{\Fr} W_{\pi}(\O_{\c_F})$ is an automorphism.

\begin{lemma}
The ring homomorphism $\Fr$ of Definition~\ref{def: Frob} is an automorphism.
\end{lemma}

\begin{proof}
It's clear that the left-shift map
\[
(\underline{x}_1, \underline{x}_q, \ldots) \mapsto (\underline{x}_q, \underline{x}_{q^2}, \ldots) 
\]
is a two-sided inverse to $\Fr$.  Hence the ring homomorphism $\Fr$ is bijective, hence an isomorphism.  
\end{proof}

We will now use this Witt vector inverse limit construction to describe a familiar ring from $p$-adic Hodge theory.  Recall the perfect characteristic~$p$ ring $\wt{\e}^+$ described above Corollary~\ref{e and a corollary}.  
%The corresponding ring of (classical) Witt vectors $W(\wt{\e}^+)$ is denoted $\wt{\a}^+$.  

\begin{proposition}[{\cite[Proposition~4.5]{DavisKedlaya}}] \label{prop: Fontaine}
There are canonical isomorphisms 
\[ 
\varprojlim_{\Fr} W_{\pi}(\O_{\c_F}) \stackrel{\varpi}{\rightarrow} \varprojlim_{\Fr} W_{\pi}(\O_{\c_F} / (\pi)) 
\stackrel{\alpha}{\rightarrow} W_{\pi}( \varprojlim \O_{\c_F} /(\pi)) \stackrel{\sim}{\rightarrow} W(\wt{\e}^+) \otimes_{W(k_F)} \O_F=:\wt{\a}^+_F.
\]
Here the transition maps in $\varprojlim \O_{\c_F}/(\pi)$ are the $q$-power Frobenius, and only the Witt vectors in the right-most term are the classical $p$-typical Witt vectors; all other Witt vectors are the generalized Witt vectors as in Definition~$\ref{Witt variant}$.
\end{proposition}

Note that the statement that these isomorphisms are \emph{canonical} includes equivariance with respect to the action of the absolute
Galois group of $F$ on every term.

\begin{proof}
The ring $\varprojlim \O_{\c_F}/(\pi)$ is isomorphic to $\wt{\e}^+$ ({\em cf.} the proof of
Proposition \ref{independent of mfa}), which is a perfect $k_F$-algebra of characteristic $p$.  Hence the final isomorphism follows from Proposition~\ref{perfect k-algebra}.  Thus we concentrate on the first two maps.

The map $\varpi$ is induced by functoriality
of Witt vectors. We check that it is injective.   Write an element $\underline{x} \in \varprojlim W_{\pi}(\O_{\c_F})$ as $(\underline{x}_{q^i})$, where $\underline{x}_{q^i} \in W_{\pi}(\O_{\c_F})$ and where $\Fr^n(\underline{x}_{q^{i+n}}) = \underline{x}_{q^i}$.  
Using the definition of the Witt vector Frobenius, we obtain the equation
\begin{equation} \label{pi kernel}
x_{q^i 0} = \sum_{j=0}^{n} \pi^jx_{q^{i+n}j}^{q^{n-j}} \qquad (i,n \geq 0). 
\end{equation}
Suppose now that $\varpi(\underline{x}) = 0$.  This implies that $v(x_{q^i j}) \geq 1$ for all $i,j \geq 0$, where we write $v$ for the $\pi$-adic valuation on $\O_{\c_F}$.  
By \eqref{pi kernel}, for all $i,n \geq 0$, we have 
$v(x_{q^i 0}) \geq n$ because $j + q^{n-j} \geq n$ for $0 \leq j \leq n$.
Hence $x_{q^i 0} =
0$ for all $i \geq 0$. If for some $n$ we have $x_{q^i j} = 0$ for all $i$ and all $j < n$, then from \eqref{pi kernel} we immediately obtain $x_{q^i n} = 0$
for all $i \geq 0$. We thus conclude that $x_{q^i j} = 0$ for all $i,j \geq 0$, 
so $\varpi$ is injective.

To see that $\varpi$ is surjective, we construct a preimage of 
$\underline{x} \in \varprojlim_{\Fr} W_{\pi}(\O_{\c_F}/(\pi))$. For each $i,j \geq 0$,
choose any lift $y_{q^i j} \in \O_{\c_F}$ of $x_{q^i j} \in \O_{\c_F}/(\pi)$,
and put $\underline{y}_{q^i} = (y_{q^i 0}, y_{q^i 1}, \dots) \in W_{\pi}(\O_{\c_F})$.
Using the polynomials expressing Frobenius in terms of Witt components (Lemma~\ref{Frobenius polynomials}), one can check
that for each $i,j \geq 0$,  as $k \to \infty$,
the $j$-th Witt component of
$\Fr^k(\underline{y}_{q^{i+k}})$ converges $\pi$-adically
to some limit $z_{q^i j}$. These define an element $\underline{z} \in \varprojlim_{\Fr} W_{\pi}(\O_{\c_F})$
with $\varpi(\underline{z}) = \underline{x}$.

Having proved that $\varpi$ is an isomorphism, we now consider the map $\alpha$, which we must first define.  For $\underline{x} \in \varprojlim_{\Fr} W_{\pi}(\O_{\c_F} /(\pi))$,
the sequence $\underline{y}_{i} = (x_{1 i}, x_{q i}, x_{q^2 i}, \dots)$ defines an element of
$\varprojlim \O_{\c_F}/(\pi)$ because the Witt vector Frobenius on $W(\O_{\c_F}/(\pi))$ is the map which in each coordinate sends $x \mapsto x^q$.  Using the polynomials defining the ring operations on generalized Witt vectors \cite[\S{1.19}]{BorgerGeometry},
we can check that setting $\alpha(\underline{x}) = (\underline{y}_0, \underline{y}_1, \dots)$ in $W(\varprojlim \O_{\c_F}/(\pi))$
defines a ring homomorphism.
Finally,  the map $\alpha$ is clearly injective and surjective, as it simply involves permuting certain indices.  This completes the proof.
\end{proof}

\begin{remark} \label{weak topology remark}
Recall that the \emph{weak topology} on $W(\wt{\e}^+)$ is obtained by identifying $W(\wt{\e}^+)$ with $\prod \wt{\e}^+$ (as a set), and equipping the latter with the product topology, where
each factor is given its valuation topology. In terms of $\varprojlim W_{\pi}(\O_{\c_F})$, this corresponds to the coarsest topology such that the maps 
\[
\varprojlim_{\Fr} W(\O_{\c_F}) \rightarrow \varprojlim_{x \mapsto x^q} \O_{\c_F}/(\pi), \qquad (\underline{x}_{1}, \underline{x}_q, \underline{x}_{q^2}, \ldots ) \mapsto (\overline{x}_{1j}, \overline{x}_{qj}, \overline{x}_{q^2j},\ldots) 
\]
(which are not ring maps if $j > 0$) are continuous for all $j \geq 0$.  
\end{remark}

Recall that we write $q = p^a$.  Using the fact that modulo~$\pi$, the Witt vector Frobenius is the same as raising each coordinate to the $q$-th power, one can check that our map $\Fr$ defined in Definition~\ref{def: Frob} agrees with the map $\Fr^a \otimes 1$ on $W(\wt{\e}^+) \otimes_{W(k_F)} \O_F$, via the isomorphism in Proposition~\ref{prop: Fontaine}.

\begin{definition} \label{def: beta}
For an element $x_{q^i j} \in \O_{\c_F}$, write $\overline{x}_{q^i j}$ for its image in $\O_{\c_F}/(\pi)$.  Let $\beta$ denote the ring homomorphism
\[
\beta: \varprojlim_{\Fr} W_{\pi}(\O_{\c_F}) \rightarrow \varprojlim_{x \mapsto x^q} \O_{\c_F}/(\pi), \qquad \underline{x} \mapsto (\overline{x}_{11}, \overline{x}_{q1}, \overline{x}_{q^2 1}, \ldots).
\]
\end{definition}

\begin{remark}\label{betamapexplicit}
Under the isomorphism 
\[
\varprojlim_{\Fr} W_{\pi}(\O_{\c_F}) \cong W(\wt{\e}^+) \otimes_{W(k_F)} \O_F=:\wt{\a}^+_F.
\]
from Proposition~\ref{prop: Fontaine}, the ring homomorphism $\beta$ corresponds to reduction modulo the ideal generated by $1 \otimes \pi$.  In terms of $W_{\pi}(\wt{\e}^+)$, the map $\beta$ corresponds to projection onto the first Witt component (or equivalently, onto the first ghost component).
\end{remark}

Proposition~\ref{prop: Fontaine} allows us to relate $\varprojlim_{\Fr} W_{\pi}(\O_{\c_F})$ to $\wt{\a}^+ = W(\wt{\e}^+)$.  One pleasant feature of this comparison is that it provides a simple description of the $\theta$ map from $p$-adic Hodge theory.  We first recall this map.

\begin{definition} \label{theta for ramified}
Define the $\theta$ map 
\[
\theta: W \left( \varprojlim_{x \mapsto x^p} \O_{\c_F}/(\pi) \right) \rightarrow \O_{\c_F}, \qquad \sum p^i [a_i] \mapsto p^i \lim_{j \rightarrow \infty} \wt{a_{ij}}^{p^j},
\]
where $a_{ij}$ is the $j$-th term in the inverse system corresponding to $a_i$, and where $\wt{a_{ij}}$ is any lift from $\O_{\c_F}/(\pi)$ to $\O_{\c_F}$.  If we replace the transition maps $x \mapsto x^p$ by $x \mapsto x^q$, the map becomes
\[
\theta: \sum p^i [a_i] \mapsto p^i \lim_{j \rightarrow \infty} \wt{a_{ij}}^{q^j}.
\]
For ramified Witt vectors, this becomes
\[
\theta: W \left( \varprojlim_{x \mapsto x^q} \O_{\c_F}/(\pi) \right) \otimes_{W(k_F)} \O_F \rightarrow \O_{\c_F}, \qquad \sum \pi^i [a_i] \mapsto \sum \pi^i \lim_{j \rightarrow \infty} \wt{a_{ij}}^{q^j}.
\]
\end{definition}

\begin{proposition}
For an element $\underline{x} \in \varprojlim_{\Fr} W_{\pi}(\O_{\c_F})$, write $(\underline{x}_1, \underline{x}_q, \ldots, \underline{x}_{q^i}, \ldots)$ for the Witt vectors in this inverse system and write $x_{q^i j}$ for the $j$-th Witt component of the Witt vector $\underline{x}_{q^i}$.  Define $\theta$ to be the ring homomorphism
\[
\theta :\varprojlim_{\Fr} W_{\pi}(\O_{\c_F}) \rightarrow \O_{\c_F}, \qquad \underline{x} \mapsto x_{q^{0}0}.
\]
Under the isomorphisms of Proposition~$\ref{prop: Fontaine}$, this induces the same map $\theta$ as in Definition~$\ref{theta for ramified}$.  
\end{proposition}  

\begin{proof}
We adapt the argument \cite[Proposition~5.8]{DavisKedlaya} which is due to Ruochuan Liu.  Assume 
\[
\underline{a} = \sum \pi^i [a_i] \in W(\wt{\e}^+) \otimes_{W(k_F)} \O_F
\]
corresponds to an element $\underline{x} \in \varprojlim_{\Fr} W_{\pi}(\O_{\c_F})$ under the isomorphisms of Proposition~\ref{prop: Fontaine}.  Recall that we let $\wt{a_{ij}}$ denote an arbitrary lift of $a_{ij} \in \O_{\c_F}/(\pi)$ to $\O_{\c_F}$.  To prove the proposition, we will describe
\[
\theta(\underline{a}) = \sum \pi^i \lim_{j \rightarrow \infty} \wt{a_{ij}}^{q^j}
\]
in terms of the coordinates of $\underline{x} \in \varprojlim_{\Fr} W_{\pi}(\O_{\c_F})$.  If we view $\underline{a}$ as an element of $W_{\pi}\left(\varprojlim_{x \mapsto x^q} \O_{\c_F}\right)$, then by Corollary~\ref{mult by pi char p} it has coordinates $(a_0, a_1^q, a_2^{q^2}, \ldots)$.  By definition of the maps $\varpi, \alpha$ in Proposition~\ref{prop: Fontaine}, we find that $x_{q^j i}$ is a lift of $a_{ij}^{q^i}$.  Then we can take $x_{q^{j+i} i}$ for $\wt{a_{ij}}$, our lift of $a_{ij}$.  We then compute
\begin{align*}
\sum_{i = 0}^{\infty} \pi^i \lim_{j \rightarrow \infty} \wt{a_{ij}}^{q^j} &= \lim_{j \rightarrow \infty} \sum_{i = 0}^{j} \pi^i \wt{a_{ij}}^{q^j} \\
&= \lim_{j \rightarrow \infty} \sum_{i = 0}^{j} \pi^i x_{q^{j+i} i}^{q^j} \\ 
&= \lim_{j \rightarrow \infty} \sum_{i = 0}^{j} \pi^i x_{q^{j} i}^{q^{j-i}}.
\intertext{The term $\sum_{i = 0}^{j} \pi^i x_{q^{j} i}^{q^{j-i}} \in \O_{\c_F}$ is the $j$-th ghost component of the Witt vector $\underline{x}_{q^j}$.  If we write $w_j$ for this ghost component, then the above limit becomes}
&= \lim_{j \rightarrow \infty} w_j(\underline{x}_{q^j}).
\end{align*}
Because $\underline{x} \in \varprojlim_{\Fr} W_{\pi}(\O_{\c_F})$, we have $\Fr(\underline{x}_{q^{j}}) = \underline{x}_{q^{j-1}}$.  By the definition of the Frobenius map (Definition~\ref{Frobenius definition}), we find that $w_{j}(\underline{x}_{q^j})$ is independent of $j$.  Taking $j = 0$ completes the proof.
\end{proof}

\section{Norm rings} \label{Sec: norm}

Let $L/K$ denote an arbitrary strictly APF extension.  Recall \cite[1.4]{Wintenberger} that 
canonically attached to $L/K$ is the sequence of ramification breaks $\{b_m\}_{m > 0}$:
this is the increasing sequence of real numbers $b$ for which $G_K^{b+\varepsilon}G_L\neq G_K^bG_L$
for all $\varepsilon > 0$.  Let $K_0$ be the maximal unramified extension of $K$ contained in $L$,
and for $m>0$ denote by $K_m$
%:={\Kbar}^{G_K^{b_m}G_L}$ 
the fixed field of $G_K^{b_m}G_L$ acting on $\Kbar$; it is
a subfield of $L$ of finite degree over $K$ with the property
that $K_{n+1}/K_n$ is {\em elementary of level $i_n:=\Psi_{L/K}(b_n)$},
where $\Psi_{L/K}$ is the transition function of Herbrand as in \cite[1.2.1]{Wintenberger}.
Following \cite[1.4]{Wintenberger}, we call $\{K_m\}_{m\ge 0}$ the {\em tower of elementary
extensions} associated to $L/K$, and for $m\ge 1$ we write $r_m$ for the unique positive integer
determined by $p^{r_m}=[K_m:K_1]$.
 
%$L$ is 
%canonically the rising union of the associated {\em elementary tower}
%of extensions $\{K_m\}_m$.  
%Each $K_m$ is a finite extension of $K$, with $K_0/K$ and $K_1/K$
%the maximal unramified and tamely ramified sub-extensions of $L/K$, respectively,
%so $K_m/K_1$ is totally wildly ramified of degree $p^{r_m}:=[K_m:K_1]$.
%Moreover, $K_m/K_{m-1}$ is {\em elementary of level $i_m$} in the sense
%of \cite[1.3]{Wintenberger}.

\begin{lemma}\label{Autaction}
	Each $K_m$ is preserved by $\Aut(L/K)$.  In particular,
	when $L/K$ is Galois, so is $K_m/K$.
\end{lemma}

\begin{proof}
	Fix $\sigma\in \Aut(L/K)$, and let $\wt{\sigma}\in G_K$ be any extension of $\sigma$
	to an automorphism of $\Kbar$.  Then 
	$${\wt{\sigma}}^{-1} G_K^{b_m}G_L \wt{\sigma} = ({\wt{\sigma}}^{-1} G_K^{b_m} \wt{\sigma})(
	{\wt{\sigma}}^{-1} G_L\wt{\sigma}) = G_K^{b_m} G_L$$
	as $G_K^{b_m}$ is normal in $G_K$ and $G_L$ is stable under conjugation
	by any automorphism of $\Kbar$ that preserves $L$.  The result follows.  
\end{proof}

Let us recall the definition of the {\em perfect norm field} associated to $L/K$:

\begin{proposition} \label{independent of mfa}
		Let $\mfa\subseteq \O_L$ be any ideal 
%		such that $\mfa^N\subseteq \b\subseteq \mfa$ for some integer $N\ge 1$. 
		with the property that the 
		$\mfa$-adic and $p$-adic topologies on $\O_L$ coincide.  
		Then reduction
		modulo $\mfa$ induces a multiplicative bijection
		\begin{equation}
		\xymatrix{
			{\varprojlim_{x\mapsto x^p} \O_{\wh{L}}}\ar[r]^-{\simeq} & 
			{\varprojlim_{x\mapsto x^p} \O_{L}/\mfa}.
			}\label{standard}
		\end{equation}
		In particular, $\wt{\e}_{L/K}^+:={\varprojlim_{x\mapsto x^p} \O_{L}/\mfa}$ is naturally a subring of 
		$\wt{\e}^+$ that is independent of $\mfa$.
	\end{proposition}
	
%	\bryden{There is a small wrinkle to iron out here: it may not be the case
%	that $\mfa=(p)$ is allowed, so the claimed embedding into $\wt{\e}^+$ is not obvious.
%	This needs to be sorted out.}
	
	\begin{proof}
		This is standard ({\em e.g.} \cite[Proposition~4.3.1]{BC09}): if $(x_n)_n\in \wt{\e}^+_{L/K}$,
		then for any choices of lifts $\wh{x}_n\in \O_L$ of $x_n$, our hypotheses
		on $\mfa$ ensure that the $\mfa$-adic and $p$-adic topologies on $\O_{\wh{L}}$
		coincide, so for each fixed $n$, the sequence $\{\wh{x}_{n+m}^{p^m}\}_m$
		converges to an element $y_n$ of $\O_{\wh{L}}$ that is independent
		of these choices with the resulting sequence $\{y_n\}_n$ compatible
		under the $p$-power map.  The association $x_n\rightsquigarrow y_n$
		gives the desired inverse mapping to (\ref{standard}).
		The rest of the Proposition follows easily.
	\end{proof}

The {\em strictness} hypothesis on the APF extension $L/K$ amounts to the
assumption that the quantities $i_n/[K_{n+1}:K_n]$ are bounded below by a positive constant,
and we define the constants \cite[1.2.1, 1.4.1]{Wintenberger}
\begin{equation}
	c(L/K_m):=\inf_{n \ge m} \frac{i_n}{[K_{n+1}:K_m]}.\label{cconst}
\end{equation}
Denote by $\b$ the ideal of $\O_L$ defined by
\begin{equation}
	\b:=\{\alpha\in \O_L\ :\ v_{K_1}(\alpha) \ge c(L/K_1)\}.\label{bideal}
\end{equation}
If $\mfa$ is any ideal of $\O_L$, then for any extension $E$ of $K_1$
contained in $L$ 
we write $\mfa_E:=\mfa\cap \O_E$ for the 
induced ideal of $\O_E$, which is visibly the kernel of the canonical map $\O_E\hookrightarrow \O_L\twoheadrightarrow
\O_L/\mfa$.  For any subfields $E'\supseteq E$ of $L$ containing $K$, we may (and do) therefore view $\O_E/\mfa_E$ as a subring of $\O_{E'}/\mfa_{E'}$.

\begin{definition}\label{normring}
	Let $L/K$ be an infinite and strictly APF extension, and let $\mfa\subseteq \O_L$ be any ideal such that $\mfa^N\subseteq \b\subseteq \mfa$ for some integer $N\ge 1$.   
	We define 
	%\begin{equation*}
	%	\wt{\e}_{L/K}^+(\mfa) = \varprojlim_{x\mapsto x^p} \O_L/{\mfa}
	%\end{equation*}
	%and
	\begin{equation*}
		\e_{L/K}^+(\mfa) := \left\{
			(x_n)\in \varprojlim_{x\mapsto x^p} \O_L/\mfa\ :\ 
			x_{r_m}\in \im\left(\O_{K_m}/\mfa_{K_m}\hookrightarrow \O_L/\mfa\right)\  \text{for all}\ m 
		\right\}
%		\e_{L/K}^+ :=
%		\{ (x_n)_n\in \varprojlim_{x\mapsto x^p} \O_{\c_F}/(p)\ : x_n\in \O_{E}/(p)\ \text{whenever}\ 
%		p^n | [E:K]\ \text{and}\ n\ge N(L/K)\},
	\end{equation*}
	For ease of notation, we set %$\wt{\e}^+_{L/K}:=\wt{\e}^+_{L/K}(\b)$ and 
	$\e_{L/K}^+:=\e_{L/K}^+(\b)$ and $\e_{L/K}:=\Frac(\e_{L/K}^+)$.
\end{definition}

	By definition, $\e_{L/K}^+$ is a subring of $\wt{\e}_{L/K}^+$---depending only on $L/K$---which, thanks to Lemma 
	\ref{Autaction}, is stable under the natural coordinate-wise action of $\Aut(L/K)$
	on $\wt{\e}_{L/K}^+$.    

	Let $X_K(L)$ be the imperfect norm field attached to $L/K$ as in \cite[\S{2.1}]{Wintenberger}.
	For each finite intermediate extension $E/K_1$, define 
	\begin{equation}
		r(E):=\left\lceil \frac{(p-1)}{p}i(L/E)\right\rceil\qquad\text{where}\qquad
		i(L/E):= \sup \{i\ge -1\ :\ G_E^iG_L = G_E\}.
	\end{equation}
	Due to \cite[2.2.3.1]{Wintenberger}, the integers $r(E)$ are nondecreasing and tend to infinity
	with respect to the directed set $\E_{L/K_1}$ of intermediate extensions $L\supseteq E\supseteq K_1$
	that are finite over $K_1$.
	Let $s:=\{s(E)\}_{E\in \E_{L/K_1}}$ be any fixed choice
	of positive integers $s(E)\le r(E)$ with $s(E)$ nondecreasing and tending to infinity
	(with respect to $\E_{L/K_1}$). Thanks to \cite[2.2.1, 2.2.3.3]{Wintenberger},
	the norm maps induce {\em ring homomorphisms} 
	$\Nm_{E'/E}: \O_{E'}/\mathfrak{m}_{E'}^{s(E')}\rightarrow \O_{E}/\mathfrak{m}_{E}^{s(E)}$
	by  and we put
	\begin{equation}
		A_K(L)_{s} := \varprojlim_{E\in \E_{L/K_1}} \O_E/\mathfrak{m}_E^{s(E)}.
	\end{equation}
	%and denote by $A_K(L)$ the ring defined by \cite[2.2.3.3]{Wintenberger}, which 
	%by Proposition 2.3.1 and 2.2.3.3 of \cite{Wintenberger} is identified with the valuation ring of
	%$X_K(L)$.
	
	\begin{lemma}\label{sindep}
		Let $r:=\{r(E)\}_{E\in \E_{L/K_1}}$ and let $s:=\{s(E)\}_{E\in \E_{L/K_1}}$
		be any choice of positive integers as above.  Then the natural reduction map
		$A_K(L)_r\rightarrow A_K(L)_s$ is an isomorphism of rings; in particular,
		the ring $A_K(L):=A_K(L)_s$ is independent of the choice of $s$.  Moreover,
		$A_K(L)$ is canonically identified with the valuation ring of the norm field $X_K(L)$.
	\end{lemma}
	  
	\begin{proof}
	Given a norm compatible sequence $(x_E)_E \in \varprojlim \O_{E}/\mathfrak{m}_{E}^{s(E)}$,
	we choose for each $n$ an arbitrary lift $\wh{x}_E\in \O_{E}$ of $x_E$, and we set 
	$y_E:=\varinjlim_{E'\in \E_{L/E}} \Nm_{E'/E}(\wh{x}_{E'})$; this limit exists and is independent
	of the choices of lifts $\wh{x}_{E'}$ by the proof 2.3.2 of \cite[Proposition 2.3.1]{Wintenberger}.
	By construction, $(y_E\bmod \mathfrak{m}_{E}^{r(E)})_E$ is then a norm compatible sequence in 
	$\varprojlim_{\Nm} \O_{E}/\mathfrak{m}_{E}^{r(E)}$ lifting the given sequence $(x_E)_E$,
	and the association $(x_E)_E\mapsto (y_E)_E$ provides the desired inverse to the reduction
	map $A_K(L)_r\rightarrow A_K(L)_s$; see \cite[\S2.3]{Wintenberger} for further details.
	That $A_K(L)$ is canonically identified with the valuation ring of $X_K(L)$ is 
	\cite[Proposition 2.3.1]{Wintenberger}.
	\end{proof}

	By \ref{cconst}, we have
	$$i(L/K_n)=i(K_{n+1}/K_n) = [K_{n+1}:K_1]\frac{i(K_{n+1}/K_n)}{[K_{n+1}:K_1]} \ge [K_{n+1}:K_1]c(L/K_1),$$
	and it follows that $r(K_n) \ge \lceil(p-1) [K_n:K_1] c(L/K_1)\rceil$.  In particular,
	the sequence of positive integers $s(K_n):=\lceil [K_n:K_1] c(L/K_1)\rceil$ has $s(K_n)\le r(K_n)$,
	is nondecreasing and tends to infinity with $n$.  Furthermore, 
	as the value group of $v_{K_n}$ is $\Z$, we have by definition (\ref{bideal}) of $\b$	
	\begin{equation*}
		\b_{K_n} = \{x\in \O_{K_n}\ :\ v_{K_n}(x) \ge [K_n:K_1]c(L/K_1)\} = \mathfrak{m}_{K_n}^{s(K_n)}.
	\end{equation*}
	If $\mfa\subseteq \O_L$ is any ideal with $\mfa^N\subseteq \b\subseteq \mfa$, then
	for any $E\in \E_{L/K_1}$ 
	we have $\mfa_{E}^N\subseteq (\mfa^N)_{E}\subseteq \b_E\subseteq \mfa_E$,
	so defining integers $s_{\mfa}(E)$ by $\mfa_E=\mathfrak{m}_E^{s_{\mfa}(E)}$ we have
	$s_{\mfa}(K_n)\le s(K_n) \le Ns_{\mfa}(K_n)$ and the sequence $\{s_{\mfa}(K_n)\}_n$
	tends to infinity with $n$ and has $s_{\mfa}(K_n) \le r(K_n)$.  We claim that 
	the sequence $s_{\mfa}(K_n)$ is nondecreasing. 
To see this, let $E'/E$ denote any two fields in $\E_{L/K_1}$.  We then have
\[
\mathfrak{m}_E^{s_{\mfa}(E)} = \mfa_E = \mfa_{E'} \cap \O_E = \mathfrak{m}_{E'}^{s_{\mfa}(E')} \cap \O_E \supseteq (\mathfrak{m}_{E'} \cap \O_E)^{s_{\mfa}(E')} = \mathfrak{m}_E^{s_{\mfa}(E')};
\]
comparing the left and right sides shows that $s_{\mfa}(E) \leq s_{\mfa}(E')$, as desired.

%\bryden{Is there an easier way to 
%	see this than the following?}  Indeed, if $\alpha\in \mfa_{K_{n+1}}$ has
%	$v_{K_{n+1}}(\alpha)=s_{\mfa}(K_{n+1})$, then on the one hand $\beta:=\Nm_{K_{n+1}/K_n}(\alpha)\in \mfa_{K_n}$.
%	On the other hand, for any $\alpha\in \O_{K_{n+1}}$ we have
%	\begin{equation}
%			v_{K_n}(\Nm_{K_{n+1}/K_{n}}(\alpha) - \alpha^{[K_{n+1}:K_{n}]}) \ge c(K_{n+1}/K_n) \ge c(L/K_n) \ge
%			c(L/K_1) [K_n:K_1],
%			\label{valest}
%	\end{equation}
%	with the first two inequalities given by 4.2.2.1 and 1.2.3 (\rmnum{4}) of \cite{Wintenberger}, 	respectively, and the final inequality
%	following immediately from the definition (\ref{cconst}) of $c(L/K_m)$.
%	It follows that we may write $\beta = \alpha^{[K_{n+1}:K_n]}+\gamma$, with 
%	$v_{K_n}(\gamma)\ge s(K_n)\ge s_{\mfa}(K_n)$.  Note also that, from our choice of $\alpha$,  we have 	$v_{K_n}(\alpha^{[K_{n+1}:K_n]})=s_{\mfa}(K_{n+1})$.  If the inequality $s_{\mfa}(K_{n+1}) < s_{\mfa}(K_n)$
%	held, we would conclude that $v_{K_n}(\beta) = s_{\mfa}(K_{n+1})$, a contradiction in 
%	view of $\beta \in \mfa_{K_n}=\mathfrak{m}_{K_n}^{s_{\mfa}(K_n)}$. \chris{If you believe my proof, I think you can delete from your comment until this comment.}

	We thus obtain natural maps of rings
\begin{equation}	
	\xymatrix@C=35pt{
	{\varprojlim\limits_{E\in \E_{L/K},\Nm} \O_E/\mathfrak{m}_E^{r(E)}}\ar[r]^-{\text{forget}} & 
	{\varprojlim\limits_{n,\Nm} \O_{K_n}/\mathfrak{m}_{K_n}^{r(K_n)}} 
	\ar[r]^-{\bmod \mathfrak{m}_{\star}^{s_{\mfa}(\star)}} & {\varprojlim\limits_{n,\Nm} \O_{K_n}/\mathfrak{m}_{K_n}^{s_{\mfa}(K_n)}=
	\varprojlim\limits_{n,\Nm} \O_{K_n}/\mfa_{K_n}}
	}\label{altdesc}
\end{equation}

\begin{lemma}\label{prelimisom}
	The maps $(\ref{altdesc})$ are ring isomorphisms.
\end{lemma}

\begin{proof}
	The first is an isomorphism as $\{K_n\}$ is cofinal in the directed system $\E_{L/K_1}$
	of intermediate extensions $L/E/K$.
	That the second is an isomorphism follows from the proof of Lemma~\ref{sindep}.
	%showing that $A_K(L)$ is independent of the choice of sequence $\{s(E)\}_{E\in \E_{L/K_1}}$.
	\end{proof}

\begin{corollary} \label{comparison with Wintenberger's norm ring}
	For any ideal $\mfa\subseteq \O_L$ with $\mfa^N\subseteq \b\subseteq \mfa$
	for some integer $N$, 
	there is a canonical isomorphism of rings $A_K(L)\simeq \e_{L/K}^+(\mfa)$.
	In particular, $\e_{L/K}^+(\mfa)\simeq \e_{L/K}^+$ and these rings all correspond to the same subring of $\wt{\e}^+$.  
\end{corollary}

\begin{proof}
	As the sequence $\{r_m\}$ is cofinal in $\N$, the forgetful map induces an isomorphism
	of rings 
	\begin{equation}
		\e_{L/K}^+(\mfa)\simeq \{(x_n)\ :\ x_n\in \O_{K_n}/\mfa_{K_n},\ \text{and}\ x_n^{[K_n:K_{n-1}]}=x_{n-1}\ \text{inside}\ 
		\O_{L}/\mfa\}
		\label{altdesc2}.
	\end{equation}
	On the other hand, we claim that $\Nm_{K_{n+1}/K_{n}}(\alpha) - \alpha^{[K_{n+1}:K_{n}]}\in \b \subseteq \mfa$ for all $\alpha\in \O_{K_{n+1}}$, so that the right sides of (\ref{altdesc}) and (\ref{altdesc2})
	are equal.  In other words, we claim that
\[
v_{K_1}\left( \Nm_{K_{n+1}/K_{n}}(\alpha) - \alpha^{[K_{n+1}:K_{n}]} \right) \geq c(L/K_1).
\]
This follows from 
\[
			v_{K_n}\left(\Nm_{K_{n+1}/K_{n}}(\alpha) - \alpha^{[K_{n+1}:K_{n}]}\right) \ge c(K_{n+1}/K_n) \ge c(L/K_n) \ge
			c(L/K_1) [K_n:K_1],
\]
with the first two inequalities given by 4.2.2.1 and 1.2.3 (\rmnum{4}) of \cite{Wintenberger}, respectively, and the final inequality
	following immediately from the definition (\ref{cconst}) of $c(L/K_m)$.  This completes the proof of the claim.  

We have subrings
\[
\e_{L/K}^+(\mathfrak{b}) \subseteq \e_{L/K}^+(\mathfrak{a}) \subseteq \wt{\e}^+,
\]
with the first inclusion given by reduction modulo~$\mathfrak{a}$ in each coordinate.  On the other hand, as observed in the previous paragraph, these first two rings are equal to the rings $A_K(L)_{s_{\mathfrak{b}}}$ and $A_K(L)_{s_{\mathfrak{a}}}$, respectively.  We then in fact have
\[
\e_{L/K}^+(\mathfrak{b}) = \e_{L/K}^+(\mathfrak{a}) \subseteq \wt{\e}^+,
\]
 by Lemma~\ref{prelimisom}.
%	To see our claim, we simply note that
%	\begin{equation}
%			v_{K_n}(\Nm_{K_{n+1}/K_{n}}(\alpha) - \alpha^{[K_{n+1}:K_{n}]}) \ge c(K_{n+1}/K_n) \ge c(L/K_n) \ge
%			c(L/K_1) [K_n:K_1],
%			\label{valest}
%	\end{equation}
%	with the first two inequalities given by 4.2.2.1 and 1.2.3 (\rmnum{4}) of \cite{Wintenberger}, 	respectively, and the final inequality
%	following immediately from the definition (\ref{cconst}) of $c(L/K_m)$.
	%It follows that for any element $(x_n)_n\in \varprojlim_{n,\Nm} \O_{K_n}/\b_{K_n}$,
	%we have $x_{n+1}^{[K_{n+1}:K_n]}-x_n \in \b_{K_{n+1}}$, whence $x_n = x_{n+1}^{[K_{n+1}:K_n]}$
	%in $\O_{K_{n+1}}/\b_{K_{n+1}}$ and the right sides
	%of (\ref{altdesc}) and (\ref{altdesc2})
	%are equal.  
\end{proof}

%\begin{corollary}
%	Let $\mfa\subseteq \O_L$ be any ideal such that $\mfa^N\subseteq \b\subseteq \mfa$ for some integer $N\ge 1$,
%	and for any $E\in \E_{L/K}$, put $\mfa_{E}:=\mfa\cap\O_E$. Define
%	\begin{equation*}
%	\e_{L/K}^+(\mfa) := \left\{
%			(x_n)\in \varprojlim_{x \mapsto x^p} \O_L/\mfa  :\ 
%			x_{r_m}\in \im\left(\O_{K_m}/\mfa_{K_m}\hookrightarrow \O_L/\mfa\right)\ \text{for all}\ m 
%		\right\}.
%	\end{equation*}
%	The natural reduction map $\e_{L/K}^+ = \e_{L/K}^+(\b)\rightarrow \e_{L/K}^+(\mfa)$
%	is an isomorphism of rings.
%\end{corollary}

%\begin{remark}\label{concrete}
	If $K'\in \E_{L/K}$ if any finite extension of $K$ contained in $L$, then
	as the finite subextensions
	of $L/K'$ and $L/K$ are co-final among each other,
	there is a canonical isomorphism of abstract rings $A_K(L)\simeq A_{K'}(L)$
	{\em cf.} \cite[2.1.4]{Wintenberger}.
	We warn the reader, however, that the embedding of $A_K(L)$
	into $\wt{\e}^+$ described in \cite[4.2]{Wintenberger}---whose image is exactly $\e_{L/K}^+$---is sensitive to $K$. 
	More precisely:
	
\begin{proposition}\label{pinsep}
Let $K'$ be a finite extension of $K$ contained in $L$.  Let $K_1$ $($resp., $K'_1$$)$ be the maximal tamely ramified subextension of $L/K$ $($resp., $L/K'$$)$.  
Considered as subfields of $\wt{\e}$ via the embedding described in \cite[4.2]{Wintenberger}, we have that $X_{K'}(L)$ is a purely inseparable extension of $X_{K}(L)$ of degree $[K'_1 : K_1]$.
In particular, the group $\Aut(L)=\Aut(L/\Q_p)$ acts naturally on $X_{K'}(L)$.
\end{proposition}

\begin{proof}
As abstract rings, both $X_{K}(L)$ and $X_{K'}(L)$ are isomorphic to $k(\!(x)\!)$ by \cite[Th\'eor\`eme~2.1.3]{Wintenberger}.  Let $(\pi_E)_E$ denote a uniformizer for either $X_K(L)$ or $X_{K'}(L)$; we may choose the same elements $\pi_E$ in both cases, although the index sets differ.  Then the map in \cite[4.2]{Wintenberger} sends this element to the $p$-power compatible sequence with $n$-th term
\[
\lim_{E,\,[E:K_1] >\!> 0} \pi_E^{p^{-n} [E:K_1]} \qquad \text{ or } \qquad \lim_{E,\,[E:K'_1] >\!> 0} \pi_E^{p^{-n} [E:K'_1]}.
\]
Thus we have subfields $k(\!(x)\!) \subseteq k(\!(x')\!) \subseteq \wt{\e}$, and clearly $(x')^{[K'_1:K_1]} = x.$  Because $K'_1$ is a subextension of the totally wildly ramified extension $L/K_1$, we have that $[K'_1 : K_1]$ is a power of $p$.  The final assertion follows by taking $K=\Q_p$ and $L/K'$ arbitrary strictly APF
and the fact that $\Aut(L)$ acts naturally on $\e_{L/\Q_p}$ by the preceding constructions and discussion.
This completes the proof. 
\end{proof}

\section{The ring \texorpdfstring{$\ALK$}{ALK}} \label{main section Cohen}

In Section~\ref{Witt constructions section}, we described the ring $\wt{\a}^+_F$ using inverse limits of Witt vectors.    In this section, we will use this formulation to construct canonical rings $\ALK$ sitting inside of $\wt{\a}^+_F$ for any infinite strictly APF extension $L/K$ with $K \supseteq F$.  (We remark that our ring $\ALK$ depends on $F, \pi$ as in Notation~\ref{K notation}; this dependence is suppressed from our notation.)  We show in this section that the ring $\ALK$ is always $\pi$-adically complete and separated, and always has residue ring contained within $\e_{L/K}^+$.  In the special case that $L/K$ is a $\varphi$-iterate extension (Definition \ref{iterated def}), we prove in Section~\ref{Sec:lift} that $\e_{L/K}$ is a finite purely inseparable extension of $\Frac \left(\ALK/\pi\ALK\right)$. 

The ring $\wt{\e}_{L/K}^+$ of Proposition~\ref{independent of mfa} is a perfect ring of characteristic~$p$, so the classical theory of Witt vectors suffices to lift it to a Cohen ring.  In the next proposition, we put this classical description in our ``Witt inverse limit'' framework.

\begin{proposition} \label{perfect lift}
	Let $L/K$ be an infinite strictly APF extension, 
	Then there is a canonical identification 
	\begin{equation}
		\wt{\a}_{L/K}^+:=W_{\pi}(\wt{\e}_{L/K}^+) = \varprojlim_{\Fr} W_{\pi}(\O_{\wh{L}}).
	\end{equation}
\end{proposition}

\begin{proof}
It is straightforward to adapt the proof of Proposition~\ref{prop: Fontaine}.  Note that it is important to work with $\O_{\wh{L}}$ instead of $\O_L$, because we require $\pi$-adic completeness. 
\end{proof}

\begin{remark}
	Note that $\wt{\e}_{L/K}^+$ and $\wt{\a}_{L/K}^+$ are independent of $K$
	in the sense that for any finite extension $K'/K$ contained in $L$ we have
	$\wt{\e}_{L/K'}^+=\wt{\e}_{L/K}^+$ and $\wt{\a}_{L/K'}^+=\wt{\a}_{L/K}^+$.
\end{remark}

The following is the key definition of this paper.

\begin{definition} \label{new a def}
Let $F, \pi, q$ be as in Notation~\ref{K notation}. Let $K/F$ denote a finite extension and let $L/K$ be a strictly APF extension.  Let $\{K_m\}_{m\ge 0}$ be the tower of elementary extensions of $L/K$ defined in Section~\ref{Sec: norm}.  We define
	\begin{equation*}
\begin{split}
		\ALK := \left\{(\u{x}_{q^i})_i\in \varprojlim_{\Fr} W_{\pi}(\O_{\wh{L}})\ :\ 		%\u{x}_{1}\in W(\O_K), \text{ and } \\ 
\u{x}_{q^j}\in W_{\pi}(\O_{K_m})\ \text{whenever}\ 
			q^j \mid  [K_m:K_1]		
	\right\},	
%&\forall j \geq N(L/K) \text{ and } L \supset E \supset K \text{ with } p^j \mid [E:K],\, \u{x}_{p^j} \in W(\O_E)\}.
\end{split}
	\end{equation*}
	viewed as a subring of $\wt{\a}^+_{L/K}\subseteq \wt{\a}^+_F$ via Proposition \ref{perfect lift}.
\end{definition}

\begin{remark}
As observed in Lemma~\ref{Autaction}, the elementary fields $K_m$ are stable under the action of
$\Aut(L/K)$ on $L$, so the ring $\ALK$ carries a natural action of $\Aut(L/K)$.
In particular, if $L/K$ is Galois, then $\ALK$ is a $G_K$-stable subring of $\wt{\a}_F^+$
that is fixed element-wise by $G_L$.
\end{remark}

Because $W_{\pi}(\O_E)$ is always an $\O_F$-algebra for any intermediate field $K \subseteq E \subseteq L$ and because the Witt vector Frobenius $\Fr$ is an $\O_F$-algebra homomorphism, we find that $\a_{L/K}^+$ is an $\O_F$-algebra.

Recall that the weak topology on $\wt{\a}^+ _F\cong \varprojlim W_{\pi}(\O_{\c_F})$ was described in Remark~\ref{weak topology remark}.  

\begin{proposition} \label{ALK is weakly closed} 
The $\O_F$-subalgebra $\ALK$ of $\wt{\a}^+_F$ is closed for the weak topology.
\end{proposition}

\begin{proof}
We will show that the complement is open.  Let $\underline{x} = (\underline{x}_{q^i})$ denote an element in the complement.  Then there exist some indices $i,j$ such that $q^i \mid [K_m: K_1]$ and $x_{q^i j} \not\in \O_{K_m}$.  Let $n$ be such that there is no element of $\O_{K_m}$ congruent to $x_{q^i j}$ modulo $\pi^n \O_{\c_F}$.  The collection of all elements $\underline{y} \in \varprojlim W_{\pi}(\O_{\c_F})$ such that $v_{\pi}(y_{q^i j} - x_{q^i j}) > n$ is an open set containing $\underline{x}$ and contained in the complement of $\ALK$.  This completes the proof.
\end{proof}

\begin{proposition} \label{image of beta}
Recall that $\beta: \varprojlim_{\Fr} W_{\pi}(\O_{\c_F}) \rightarrow \wt{\e}^+$ was defined in Definition~$\ref{def: beta}$.  The induced map $\beta: \ALK \rightarrow \wt{\e}^+$ has image contained in $\e_{L/K}^+$.
\end{proposition}

\begin{proof}
Let $\mathfrak{a} := (\mathfrak{b}, \pi)$.  This ideal satisfies the hypotheses of Corollary~\ref{comparison with Wintenberger's norm ring}.  By that lemma (and using its notation), it suffices to show that for $\underline{x} \in \ALK$, we have $\beta(\underline{x}) \in \e_{L/K}^+(\mathfrak{a})$.  
We will show $\beta(\underline{x})_{r_n} \in \O_{K_n}/(\pi)$ for all $n$.  Write $q = p^a$ and let $s_n$ be the smallest index such that $q^{s_n} \mid [K_n : K_1]$.  Then by the definition of $\ALK$, we know that $x_{q^{s_n} 0} \in \O_{K_n}$.  This shows that $\beta(\underline{x})_{p^{as_n}} \in \O_{K_n}/(\pi)$.  Because $r_n = as_n - d$ for some $d \geq 0$, we have that $\beta(\underline{x})_{r_n}$ is the $p^d$-th power of some element in $\O_{K_n}/(\pi)$, and hence is itself in $\O_{K_n}/(\pi)$, which completes the proof.
\end{proof}

Our next goal is to study the effect of multiplication by $\pi$ on an element in $\a_{L/K}^+$.  We will use these results to show that $\a_{L/K}^+$ is $\pi$-adically complete and separated.  Recall that Lemma~\ref{multiplication by pi} showed that for a Witt vector $\underline{x} \in W_{\pi}(R)$ which is divisible by $\pi^i$, then $x_{j}$ is divisible by $\pi^{i-j}$ for any $0 \leq j \leq i$.  
The converse to this lemma is not true; for example, the first Witt component being divisible by $\pi$ does not imply that the Witt vector is divisible by $\pi$.  For example, $V([1]) \in W_{\pi}(\O_K)$ is not divisible by $\pi$, even though its first component is zero.  However, we do have the following result (whose proof was suggested to us by Abhinav Kumar), which states that if we take a Witt vector whose first component is divisible by $\pi$, and apply Frobenius to it, then the result is a Witt vector which is divisible by $\pi$.  This result will be very useful with respect to our inverse systems under Frobenius maps.  

\begin{lemma} \label{lem: abhinav}
Let $R$ denote any $\O_F$-algebra, and let $\underline{y}\in W_{\pi}(R)$ with $y_0 \in \pi R$.  Then $\Fr(\underline{y}) \in \pi W(R)$.   
\end{lemma}

\begin{proof}
Write $\underline{y} = \pi[y'] + V(\wt{\underline{y}})$, for some $y' \in R$ and $\wt{\underline{y}} \in W_{\pi}(R)$.  Applying $\Fr$ to both sides, and using the facts that $\Fr$ is an $\O_F$-algebra homomorphism (Definition~\ref{Frobenius definition}) and that the composition $\Fr V$ is equal to multiplication by $\pi$ (Lemma~\ref{FV lemma}), immediately yields the result.   
\end{proof}

The following proposition shows that the ring $\a_{L/K}^+$ satisfies many of the desired properties which were enumerated in Theorem~\ref{MT}.  

\begin{proposition} \label{prop: W(k)-algebra etc}
Recall that $k$ denotes the residue field of $\O_L$ (or, equivalently, of $\O_{K_0}$) and that $\beta: \varprojlim_{\Fr} W_{\pi}(\O_{\c_F}) \rightarrow \wt{\e}^+$ was defined in Definition~$\ref{def: beta}$.
	The ring $\a_{L/K}^+$ is a $\pi$-torsion free, $\pi$-adically complete and separated $W_{\pi}(k)$-algebra.  Furthermore,
	the map $\beta: \a_{L/K}^+ \rightarrow \e_{L/K}^+$ has kernel equal to $\pi\a_{L/K}^+$.
\end{proposition}

\begin{proof}
We first show that $\a_{L/K}^+$ is a $W_{\pi}(k)$-algebra; this is slightly harder than showing it is a $W_{\pi}(k_F)$-algebra.  The Witt vector Frobenius $\Fr: W_{\pi}(k) \rightarrow W_{\pi}(k)$ is an $\O_F$-algebra homomorphism which induces the $q$-power map modulo $\pi$.  Hence using Proposition~\ref{lambda prop}, we have a natural ring homomorphism $\lambda_{\Fr}: W_{\pi}(k) \rightarrow W_{\pi}(W_{\pi}(k))$, and the family of ring homomorphisms 
$\lambda_{\Fr} \circ \Fr^{-n} : W_{\pi}(k) \rightarrow W_{\pi}(W_{\pi}(k))$ are Frobenius compatible, so we have an induced map to $\varprojlim_{\Fr} W_{\pi}(W_{\pi}(k))$.  From the explicit description of $W_{\pi}(k)$ given in Proposition~\ref{perfect k-algebra}, we know that we have a natural inclusion $W_{\pi}(k) \hookrightarrow \O_{K_0}$, so we have a natural map $W_{\pi}(k) \rightarrow \varprojlim_{\Fr} W_{\pi}(\O_{K_0})$.  Because $K_0 \subseteq K_n$ for all $n$, this completes the proof that $\a_{L/K}^+$ is a $W_{\pi}(k)$-algebra.

Lemma~\ref{W of pi-torsion free} shows that $\a_{L/K}^+$ is $\pi$-torsion free and Lemma~\ref{W separated} shows that $\a_{L/K}^+$ is $\pi$-adically separated.  To check that $\a_{L/K}^+$ is $\pi$-adically complete, let $\underline{x}^{(1)}, \underline{x}^{(2)}, \ldots$ denote a $\pi$-adic Cauchy sequence of elements in $\a_{L/K}^+$.   Lemma~\ref{multiplication by pi} shows that for any $i,j$, the coordinates $x_{q^ij}^{(1)}, x_{q^ij}^{(2)}, \ldots \in \O_{\wh{L}}$ form a $\pi$-adic Cauchy sequence.  The resulting limits form a Frobenius-compatible system of Witt vectors in $W_{\pi}(\O_{\wh{L}})$, because the definition of $\Fr$ in terms of coordinates involves polynomials over $\O_F$ by Lemma~\ref{Frobenius polynomials}, and such polynomials are in particular $\pi$-adically continuous.  We thus have a well-defined candidate for the limit $\underline{x} \in \wt{\a}_{L/K}^+$.   On the other hand, we know that there does exist a $\pi$-adic limit in $\wt{\a}_{L/K}^+$, because $\wt{\a}_{L/K}^+ \cong W_{\pi}(\wt{\e}_{L/K}^+),$ and the latter is $\pi$-adically complete.  Each Witt vector $\underline{x}_{q^i}$ in the inverse system corresponding to $\underline{x}$ has at least the same first Witt coordinate as the actual limit (because addition and multiplication are defined in the obvious way on the first Witt coordinate).  But Lemma~\ref{lem: abhinav} and the fact that $\wt{\a}_{L/K}^+$ is $\pi$-adically separated shows that an element in $\wt{\a}_{L/K}^+$ is uniquely determined by its sequence of first Witt components.
We have shown that our $\pi$-adic Cauchy sequence has a limit in $\wt{\a}_{L/K}^+$.  To show that the limit is in fact in $\a_{L/K}^+$, we note that each elementary subfield $K_m$ is a finite extension of $K$, and hence is $\pi$-adically complete.  

%As $\a_{L/K}^+$ is a subring of $\wt{\a}_{L/K}^+$,
%the image of $\beta: \a_{L/K}^+\rightarrow \wt{\e}^+$ is contained in 
%$\beta(\wt{\a}_{L/K}^+)$, which coincides with $\wt{\e}_{L/K}^+$
%by Remark \ref{betamapexplicit} and the identification $\wt{\a}_{L/K}^+=W_{\pi}(\wt{\e}_{L/K}^+)$
%as subrings of $\wt{\a}^+$.

Next we check that the kernel of $\beta$ is equal to $\pi\a_{L/K}^+$.  The fact that every element divisible by $\pi$ is contained in the kernel of $\beta$ follows immediately from the definitions (or from the fact that the target is characteristic~$p$).  So to prove the reverse inclusion, let $\underline{x}$ denote an element in the kernel of $\beta$.  From the fact that $\beta(\underline{x}) = 0$, we have that $x_{q^i 1}$ is divisible by $\pi$ for every $i$.   Using Lemma~\ref{lem: abhinav}, this shows that each Witt vector $\underline{x}_{q^i}$ comprising $\underline{x}$ is divisible by $\pi$.  Using that $W(\O_{\wh{L}})$ is $\pi$-torsion free by Lemma~\ref{W of pi-torsion free}, we deduce that the Witt vectors $\underline{x}_{q^i}/\pi$ are Frobenius-compatible, and hence $\underline{x}$ is divisible by $\pi$, as required.
\end{proof}

Our hope is that the ring $\ALK$ provides a functorial lift of the valuation ring in the norm field of $L/K$
in many situations.
In \S\ref{Sec:lift}, we will prove that this is indeed the case (up to a finite, purely inseparable extension) for a large
class of strictly APF extensions.  However, in many cases of interest ({\em e.g.}~when $L/K$ is Galois
with $\Gal(L/K)$ admitting no abelian quotient by a finite subgroup) our candidate lift $\ALK$ is too small.  The remainder of this section is devoted to making this statement precise.

\begin{corollary}\label{kstrict}
Write $\e^+$ and $\e$ for $\ALK/\pi\ALK$ and its fraction field, respectively.  If the residue ring $\e^+$ is strictly larger than $k$, then $\e_{L/K}$ is a finite extension of $\e.$  In particular, there is a non-canonical isomorphism of rings 
$\ALK\simeq W_{\pi}(k)[\![y]\!]$.
\end{corollary}

\begin{proof}
By \cite[Th\'eor\`eme~2.1.3]{Wintenberger} and our results from Section~\ref{Sec: norm}, we know that $\e_{L/K}^+ \cong k[\![x]\!].$  Let $y \in (x) \subseteq k[\![x]\!]$ denote an element in the image of  $\e^+$
such that $\ord_x(y)$ is minimal.   We can find a lift $\wt{y}\in\ALK$ of $y$ of the form $[y] + \pi z$, where $z \in \wt{\a}^+_{F}$.  (Note that we do not claim $z \in \ALK$.)  We have an inclusion of rings $W_{\pi}(k)\left[ \wt{y} \right] \subseteq \ALK.$  Because $\ALK \subseteq \wt{\a}^+_F$ is weakly closed (Proposition~\ref{ALK is weakly closed}), using our explicit description of $\wt{y}$, we see that in fact $W_{\pi}(k)[\![\wt{y} ]\!] \subseteq \ALK,$ which implies $k[\![y]\!] \subseteq \e^+$.  As $k(\!(y)\!) \subseteq k(\!(x)\!)$ is a finite extension of degree $\ord_x(y)$ (see for example \cite[Theorem~3]{NP06}), the first assertion follows.  

We next claim that $k(\!(y)\!) = \e$ as subrings of $\e_{L/K}$.  Consider more generally any field $k(\!(y)\!) \subseteq l \subseteq k(\!(x)\!)$.  Any element of $l \cap k[\![x]\!]$ can be written uniquely in the form
\[
a_0 + a_1 x + \cdots + a_{n-1} x^{n-1},
\]
where $n := \ord_x(y)$ and where $a_0, \ldots, a_{n-1} \in k[\![y]\!]$.  Unless $a_1 = \cdots = a_{n-1} = 0$, one may (after subtracting off $a_0$) divide by $y^m$, where $m = \min\left(\ord_y(a_1), \ldots, \ord_y(a_{n-1})\right)$ to produce an element $g$ in $l$ such that $0 < \ord_x(g) < n$.  In the specific case $l = \e$, this contradicts the minimality of $\ord_x(y)$.  This proves the claim. The embedding of rings $W_{\pi}(k)[\![\wt{y}]\!]\rightarrow \ALK$ is surjective modulo $\pi$,
and hence is an isomorphism because the target is $\pi$-adically separated. 
\end{proof}

\begin{lemma} \label{norm field faithful}
The group $\Aut(L)$ acts faithfully on $X_K(L)$.   
\end{lemma}

\begin{proof}
One adapts the proof of \cite[Lemme~3.1.3.1]{Wintenberger}: if $\tau$ is an automorphism acting trivially on $X_K(L)$, then in particular $\tau$ fixes the residue field $k$ of $L$ and also fixes uniformizers for subfields $E \subseteq L$ which are finite over $K$.  
\end{proof}

\begin{lemma}\label{faithact}
Continue to write $\e := \Frac\left( \ALK/\pi\ALK\right)$ and write $\e^{\sep}$ for the separable closure of $\e$ in $\e_{L/K}$.  The kernel of the action of $\Aut(L/K)$ on $\ALK$ has order at most $\sigma_{L/K} := [\e^{\sep} : \e]$.  
\end{lemma}

\begin{proof}
Let $H$ denote the kernel of the action of $\Aut(L/K)$ on $\ALK$.  We have natural maps
\[
H\rightarrow \Aut(\e_{L/K}/\e) \rightarrow \Aut(\e^{\sep}/\e) \rightarrow \Aut(\e).
\]
The first two maps are injective (by Lemma~\ref{norm field faithful} and the fact that $\e_{L/K}/\e^{\sep}$ is purely inseparable, respectively).  The composite map $H \rightarrow \Aut(\e)$ is clearly trivial, by our definitions of $H$ and $\e$.  It thus suffices to bound the size of the kernel of the third map, and the result follows by Galois theory.
\end{proof}

%\begin{lemma}\label{faithact}
%If $\e_{L/K}$ is a finite purely inseparable extension of $\Frac\left( \ALK/\pi\ALK\right)$, then $\Aut(L/K)$ acts faithfully on $\ALK$.  
%\end{lemma}

%\begin{proof}
%Assume $\tau \in \Aut(L/K)$ acts trivially on $\ALK$.  Then it induces the identity map on $\Frac\left( \ALK/\pi\ALK\right)$, and hence on its finite purely inseparable extension $X_K(L)$.  By Lemma~\ref{norm field faithful}, this implies that $\tau$ is trivial.  
%\end{proof}

The reduction map $\O_{L} \rightarrow k$ induces a ring homomorphism 
\[
\rho: \ALK \rightarrow W_{\pi}(\O_{L}) \rightarrow W_{\pi}(k),
\]
where the first map is projection onto the first Witt vector in the inverse system.  

\begin{lemma}\label{Frprime}
Assume that the residue ring $\e^+:=\ALK/\pi\ALK$ is strictly larger than
$k$, so that 
by Corollary \ref{kstrict}
we have a noncanonical isomorphism 
\[
\ALK \cong W_{\pi}(k)[\![y]\!].
\]
Then $\Fr(y) \in \ker \rho$ and for any unit $u \in
W_{\pi}(k)[\![y]\!]$ we have 
\[
v_{\pi} \left(\rho\left(\frac{\Fr(y)}{y}\right)\right) = v_{\pi}\left(\rho\left(\frac{\Fr(uy)}{uy}\right)\right).
\]
%In particular, writing $\Fr(y)=a_1y+\cdots$, the absolute value $|a_1|$ is well-defined $(${\em i.e.}
%depends only on $\ALK$$)$.
\end{lemma}

\begin{proof}
The first result follows from
\[
\rho\left( \Fr(y) \right) = \Fr \left( \rho(y) \right) = \Fr(0) = 0.  
\]
The second result follows from the fact that $u$ is a unit, and hence so are $\Fr(u)$ and 
$
\rho( \Fr(u)/u).
$
\end{proof}

It follows from Lemma \ref{Frprime} that if the residue ring $\e^+$ strictly contains $k$,
then we may choose an isomorphism $\ALK\simeq W_{\pi}(k)[\![y]\!]$, and
writing $\Fr(y)=a_1y+\cdots$, the absolute value $|\Fr'(0)|:=|a_1|$ is well-defined ({\em i.e.}
is independent of any choices and depends only on $\ALK$).

%\begin{definition}
%	We set $E(L/K):=\a_{L/K}^+/\pi\a_{L/K}^+$, which we identify with a subring of $\wt{\e}^+$
%	via $\beta$.  We write $v_E:E(L/K)\rightarrow \R$ for the restriction of the valuation 
%	$v_{\e}$ on $\wt{\e}^+$ to $E(L/K)$.
%\end{definition}

%\begin{proposition}\label{discrete}
%	The ring $E(L/K)$ is naturally a subring of the imperfect norm ring $\e_{L/K}^+$.
%	In particular, the valuation $v_E:E(L/K)\rightarrow \R$ is discretely valued.
%\end{proposition}

%\begin{proof}
%We have already shown that $E(L/K)$ is naturally a subring of a finite extension of $\e_{L/K}^+$ in Proposition~\ref{image of beta}.
%	Since the restriction of $v_{\e}$ to $\e_{L/K}^+$ is discrete, the same is true of $v_E$.
%\end{proof}

\begin{proposition} \label{abelian aut group proposition}
	Let $L/K$ be a strictly APF extension and suppose that the residue ring $\e^+$
	strictly contains $k$, so that $\e_{L/K}$ is a finite extension of $\e$.
	Let $\sigma_{L/K}$ be the separable degree of $\e_{L/K}/\e$,
	and suppose that $|\Fr'(0)|\neq 0$. There is a natural homomorphism
	of groups $\Aut(L/K)\rightarrow W_{\pi}(k)^{\times}$ with finite kernel
	of order at most $\sigma_{L/K}$.
	In particular, if %$|\Fr'(0)|\neq 0 $ and  
$\e_{L/K}$ is a purely inseparable extension of 
	$\e$, then $\Aut(L/K)$ is abelian. 
\end{proposition}

\begin{proof}
	%Let us write $\e^{\sep}$ for the separable closure of $\e$ in $\e_{L/K}$,
	%so that $\sigma_{L/K}=[\e^{\sep}:\e]$.  We have seen that $\Aut(L/K)$
	%acts faithfully on $\e_{L/K}$ and preserves the subfields $\e$
	%and $\e^{\sep}$.  As $\e_{L/K}/\e^{\sep}$ is finite purely inseparable,
	%the restriction map $\Aut(\e_{L/K})\rightarrow \Aut(\e^{\sep})$
	%is bijective, whence $\Aut(L/K)$ acts faithfully on $\e^{\sep}$.
	%On the other hand, an automorphism of $\e$ admits at most 
	%$\sigma_{L/K}$ extensions to an automorphism of $\e^{\sep}$,
	%so the kernel of restriction $\Aut(\e^{\sep})\rightarrow \Aut(\e)$
	%has order at most $\sigma_{L/K}$. Thus, the action map
	%$\Aut(L/K)\rightarrow \Aut(\e)$ has kernel of order at most $\sigma_{L/K}$,
	%and it follows from the proof of Lemma \ref{faithact} that 
	%the kernel $H$ of the action of $\Aut(L/K)$ on $\ALK$ 
	%has order at most $\sigma_{L/K}$.\bryden{It would be more convenient to rephrase
	%that Lemma to allow an arbitrary finite extension as in this proposition, and then I can just 
	%quote it.}
	
Lemma~\ref{faithact} shows that the kernel $H$ of the action of $\Aut(L/K)$ on $\ALK$ has order at most $\sigma_{L/K}$.
	By Corollary \ref{kstrict} and our hypotheses, we thus have a power series ring $W_{\pi}(k)[\![y]\!]$
	equipped with an endomorphism $\Fr$ determined by $\Fr(y)=f_1y+\cdots$ with $f_1\neq 0$
	and a commuting, faithful action of the group $G:=\Aut(L/K)/H$.  Put $E:=\Frac(W_{\pi}(k))$.
	Under these circumstances, we prove that
	$G$ is isomorphic to a subgroup of $W_{\pi}(k)^{\times}$ and is in particular abelian.  
	For each $g\in G$, we have $gu = a_1(g)u+\cdots$
	with $a_1(g)\neq 0$ as $G$ acts by automorphisms.  
	As in the proof of \cite[Theorem 4.1]{BergerLift},
	there is a unique power series $A(u):=a_1u+\cdots \in E[\![u]\!]$ with $a_1=1$ which satisfies
	$A(\Fr(u)) = f_1 A(u)$ ({\em cf.} \cite[Proposition 1.2]{LubinDynamics}).
	As the action of $G$ commutes with $\Fr$, we have
	\begin{equation*}
		a_1(g)^{-1}A(g\Fr(u)) = a_1(g)^{-1}A(\Fr(gu)) = f_1\left(a_1(g)^{-1}A(gu)\right)   
	\end{equation*}
	for all $g\in G$, and it follows from $a_1(g)^{-1}A(gu) = u+\cdots$ and the uniqueness of $A$
	that $$A(gu) = a_1(g)A(u).$$  From the construction of $A$ one sees
	that $A(gu)=A(u)$ if and only if $gu=u$, which happens
	if and only if $g=1$ since $G$ acts faithfully.  We conclude that the mapping
	$G\rightarrow W_{\pi}(k)^{\times}$ given by $g\mapsto a_1(g)$ is an injective homomorphism,
	whence $G$ is isomorphic to a subgroup of $W_{\pi}(k)^{\times}$ and in particular is abelian.
\end{proof}

\begin{remark}\label{pessimism}
     If $L/K$ is Galois with group that is a $p$-adic Lie group,
     then one has $|\Fr'(0)|\neq 0$ thanks to \cite[Lemma~4.5]{BergerLift}.
	It follows that if the residue ring $\ALK/\pi\ALK$ is strictly larger
	than $k$, then $\Gal(L/K)$ admits an abelian quotient by a finite subgroup. 
\end{remark}

The above results complete the proof of Theorem~\ref{MT} (\ref{MT1})-(\ref{MT3}).  The proof of Theorem~\ref{MT} (\ref{MT4}), concerning functoriality of our construction, is completed in Section~\ref{Sec:functorial}.  In Section~\ref{Sec:lift}, we prove Theorem~\ref{MT-iterate}, which says that in the case that $L/K$ is a $\varphi$-iterate extension, then the norm field $\e_{L/K}$ is a finite, purely inseparable extension of $\e$.

\section{Functoriality} \label{Sec:functorial}

The goal of this section is to prove Theorem~\ref{functoriality theorem}, which describes a sense in which the association
\[
L/K \leadsto \ALK
\]
is functorial in $L$.  This functoriality is compatible with the functoriality of $\e_{L/K}^+$.  In order to prove functoriality, we are restricted to working with finite extensions $L'/L$ which have wild ramification degree equal to a power of $q$.  

\begin{lemma} \label{intersection is elementary}
Let $L/K$ denote a strictly APF extension and let $L'/L$ denote a finite extension of degree $N$.  Let $K_n$ (respectively, $K_n'$), denote the associated elementary subfields of $L/K$ (respectively, $L'/K$).  For every $n$, there exists an $r(n)$ such that $K_n' \cap L = K_{r(n)}$.  
\end{lemma}

\begin{proof}
There exists some $u$ such that 
\[
K_n' = \textup{Fix}(G_K^u G_{L'}).
\]
Because $G_{L'} \leq G_L$, we have (for the same $u$),
\[
K_n' \cap L = \textup{Fix}(G_K^u G_{L'} G_L) = \textup{Fix}(G_K^u G_L).
\]
Every such fixed field occurs in the tower of elementary extensions of $L/K$, which completes the proof.
\end{proof}

\begin{lemma} \label{disjoint fields}
If $LK'_n = L'$, then the natural map
\begin{equation} \label{tensor equation}
L \otimes_{L \cap K'_n} K'_n \rightarrow L'
\end{equation}
is an isomorphism. 
\end{lemma}

\begin{proof}
Surjectivity of the map is clear, so we concentrate on injectivity.  It suffices to show that $L \otimes_{L \cap K'_n} K'_n$ is a field, which we prove by showing it is an integral domain, and in particular is not a direct sum of multiple fields.

Let $u$ be such that $K'_n = \textup{Fix}(G_K^u G_{L'})$.  Set $K''_n := \textup{Fix}(G_K^u)$.  Then $K''_n$ is a Galois extension of $K$, and hence is a Galois extension of any intermediate field between itself and $K$.  In particular, $K''_n$ is a Galois extension of $L \cap K''_n$.  This implies that the natural map 
\begin{equation} \label{thorup equation}
L \otimes_{L \cap K''_n} K''_n \rightarrow LK''_n
\end{equation}
is injective.   Because $L \subseteq L'$, we have 
\[
L \cap K''_n = L \cap (K''_n \cap L') = L \cap K'_n. 
\]
Using this result and flatness, we have an injective map
\[
L \otimes_{L \cap K'_n} K'_n \rightarrow L \otimes_{L \cap K'_n} K''_n = L \otimes_{L \cap K''_n} K''_n.
\]
By (\ref{thorup equation}), the right-hand side injects into a field, and hence the left-hand side $L \otimes_{L \cap K'_n} K'_n$ is an integral domain. 
\end{proof}

\begin{corollary}
For $n$ sufficiently large, the degree of $K'_n$ over $K'_n \cap L$ is independent of $n$.
\end{corollary}

\begin{proof}
Indeed, the previous lemma shows that if $n$ is sufficiently large that $LK'_n = L'$, then this degree is always the degree of $L'/L$.
\end{proof}

\begin{corollary} \label{wild ramification cor}
For any $n$, let $r(n)$ be such that $K'_n \cap L = K_{r(n)}$; this is possible by Lemma~$\ref{intersection is elementary}$.  For $n$ sufficiently large, there exists a constant $a$ independent of $n$ such that 
\[
[K'_{n}: K'_1] = a [K_{r(n)} : K_1].
\]
\end{corollary}

\begin{proof}
A calculation shows that we may take $a = [K'_{n} : K_{r(n)}] / [K'_1 : K_1]$ and we have already seen that the numerator of this is independent of $n$ for $n$ sufficiently large.
\end{proof}

The formula stated in the previous proof, together with the descriptions of $K_1$ and $K'_1$ as the maximal tamely ramified subextensions, shows the following.

\begin{corollary} \label{wild description}
The constant $a$ of Corollary~$\ref{wild ramification cor}$ is equal to the wild ramification degree of $L'/L$.
\end{corollary}

\begin{corollary} \label{n not nec large}
Let $a$ denote the same constant as in Corollary~$\ref{wild ramification cor}$.  Then for {\em all} $n\ge 1$
\[
[K'_{n} : K'_1] \leq a [K_{r(n)} : K_1].
\]
\end{corollary}

\begin{proof}
The assumption that $n$ be sufficiently large was used only only to ensure the surjectivity of (\ref{tensor equation}); the proof given in Lemma~\ref{disjoint fields} shows that this map is always injective.  This shows that 
\[
[K'_{n} : K_{r(n)}] \leq [L' : L] = a[K'_1 : K_1]
\]
for all $n$.  Using this, we have 
\[
[K'_n : K'_1] = \frac{[K'_n : K_{r(n)}] [K_{r(n)} : K_1]}{[K'_1 : K_1]} \leq a [K_{r(n)} : K_{1}],
\]
as required.
\end{proof}

\begin{theorem} \label{functoriality theorem}
If $L'/L$ has wild ramification degree a power of $q$, say $q^b$, then the map
\[
\ALK \rightarrow \a_{L'/K}^+
\]
given by 
\[
(\underline{x}_{q^i})_i \mapsto (\underline{x'}_{q^j})_j, \qquad x'_{q^{j+b}} := x_{q^{j}}
\]
lifts the norm field map $\e_{L/K}^+ \rightarrow \e_{L'/K}^+$ defined in \cite[3.1.1]{Wintenberger}.  
\end{theorem}

\begin{proof}
By Corollary~\ref{wild description}, we may write $a = q^b$, where $a$ is the constant of Corollary~\ref{n not nec large} and the surrounding results.  
We first claim that this map does have image in $\a_{L'/K}^+$.  In other words, we are given an element 
\[
(x_{q^j}) \in \varprojlim_{\Fr} W_{\pi}(\O_{\wh{L}})
\] 
such that for all $j$, if $q^j \mid [K_m : K_1]$ then $x_{q^j} \in \O_{K_m}$, and we want to show that if $q^{j+b} \mid [K'_n : K'_1]$, then $x_{q^j} \in \O_{K'_n}$.  By Corollary~\ref{n not nec large}, the condition $q^{j+b} \mid [K'_n : K'_1]$ implies $q^j \mid [K_{r(n)} : K_1]$.  Hence $x_{q^j} \in \O_{K_{r(n)}} = \O_{K'_n} \cap L \subseteq \O_{K'_n}$, as required.

We now show that the given shift map does lift the functorial map on norm fields.  By Corollary~\ref{comparison with Wintenberger's norm ring}, we may view the non-zero elements of $\e_{L/K}^+$ as norm compatible families of elements in $\O_{K_n}$ for $n$ large enough that the statement of Corollary~\ref{wild ramification cor} holds, and similarly for $L'/K$.  The description in \cite[3.1.1]{Wintenberger}, together with Lemma~\ref{disjoint fields}, shows that the norm field map is given by
\[
(z_{K_n}) \mapsto (z'_{K'_n}), \qquad z'_{K'_n} := z_{K_{r(n)}}.
\]

To complete the proof, we must relate this norm-compatible description of the norm field to the $q$-power compatible description of the norm field, as in our definition of $\e_{L/K}^+.$  If we begin with an element $(z_{K_n})$ of the norm field, viewed as a norm-compatible family, we may find a corresponding $q$-power compatible family using the map in \cite[Proposition~4.2.1]{Wintenberger}.  In particular, the $q^s$-component in the $q$-power compatible system is equal to
\[
\lim_{n \rightarrow \infty} z_{K_{r(n)}}^{q^{-s} [K_{r(n)} : K_1]} = \lim_{n \rightarrow \infty} z_{K_{r(n)}}^{q^{-s-b} [K'_n : K'_1]} = \lim_{n \rightarrow \infty} (z'_{K'_n}) ^{q^{-s-b} [K'_n : K'_1]}.
\]
This is equal to the $q^{s+b}$ component of the image of $(z_{K'_n})$.  This shows that the norm field map, in terms of $q$-power compatible systems, is given again by the ``$b$-shift'' map.  This is clearly compatible with our ``$b$-shift'' map on $\a_{L/K}^+ \rightarrow \a_{L'/K}^+$, which completes the proof.
\end{proof}

\section{Lifting a uniformizer to 
\texorpdfstring{$\ALK$}{ALK} for \texorpdfstring{$\varphi$}{phi}-iterate extensions} \label{Sec:lift}

Throughout this section, we fix a $\varphi$-iterate extension $L/K$ in the sense of Definition~\ref{iterated def}.
This definition includes the (likely redundant) hypothesis that $L/K$ is strictly APF, 
so there is a positive constant $c$ such that 
%the following assumptions.
%\begin{hypotheses} \label{hypothesis elementary}
%\begin{enumerate}
%\item $L/K$ is a $\varphi$-iterate extension in the sense of Definition~\ref{iterated def}.  
%\item \label{APF bounds hypothesis} The extension $L/K$ is a strictly APF extension.  In other words, there exists a positive constant $c$ such that
\begin{equation}
c \leq \frac{\psi_{L/K}(u)}{[G_K : G_K^{u}G_L]} 
\label{hypothesis elementary}
\end{equation}
for all $u$. 
%\end{enumerate}
%\end{hypotheses}
The goal of this section is to prove Theorem~\ref{MT-iterate}.  This theorem says loosely that, in the case of a $\varphi$-iterate extension, the map $\beta: \ALK \rightarrow \e_{L/K}^+$ is nearly surjective.  The ring $\e_{L/K}^+$ is noncanonically isomorphic to $k[\![x]\!]$, and our strategy is to construct a lift of some $q^d$-th power of the uniformizer $x$.  From our assumption that $L/K$ is a $\varphi$-iterate extension, we know explicit uniformizers $\pi_i$ for a tower of intermediate fields $K(\pi_i)$.  This enables us to construct an explicit uniformizer for $\e_{L/K}^+$.  The difficulty then is to relate these elements $\pi_i$ to the elementary extensions $K_m$ which appear in the definition of $\ALK$.  

\begin{lemma} \label{lower numbering lemma}
Let $L/K$ denote a $\varphi$-iterate extension.  
There exist positive constants $A,B$ such that for any $i \geq 1$, and any non-trivial embedding $\sigma$ of $K(\pi_i)$ into $\overline{L}$ fixing $K(\pi_{i-1})$, we have
\[
Aq^i \leq \ord_{\pi_i} (\sigma(\pi_i) - \pi_i) \leq Bq^i.
\]
\end{lemma}

\begin{proof}
We first find an $A$ such that
\[
Aq^i \leq \ord_{\pi_i} (\sigma(\pi_i) - \pi_i).
\]
We know that $\varphi(\pi_i) = \varphi(\sigma(\pi_i)),$ which shows that
\[
\ord_{\pi_i} \left((\sigma(\pi_i))^q - \pi_i^q \right) \geq \ord_{\pi_i} \pi,
\]
where $\pi$ is as in the definition of a $\varphi$-iterate extension.  On the other hand, 
\[
(\sigma(\pi_i) - \pi_i)^q \equiv \sigma(\pi_i)^q - \pi_i^q \bmod p.
\]
Thus
\[
\ord_{\pi_i} (\sigma(\pi_i) - \pi_i) \geq \frac{\ord_{\pi_i} \pi}{q}.  
\]
This shows that we may take $A = \frac{\ord_{\pi_0} \pi}{q}$.  

We now find $B$ such that for all $i$, 
\[
\ord_{\pi_i} (\sigma(\pi_i) - \pi_i) \leq Bq^i.
\]
Write $\varphi(x) = \sum a_j x^j$.  Let $j_0$ be the smallest index at which the minimum
$B_1 := \min_j \ord_{\pi_0} (j a_j)$
is attained.  (Note that $B_1 > 0$, because $a_j$ is a unit only for $j = q$.)  We then claim that for $i >\!> 0$, we have
\[
\ord_{\pi_i} \varphi'(\pi_i) = \ord_{\pi_i} (j_0 a_{j_0} \pi_i^{j_0 - 1}) = j_0 - 1 + B_1 q^{i}.   
\]  
It suffices to show that for all $j \neq j_0$ and $i >\!> 0$, we have
\[
j - 1 + \ord_{\pi_i}(ja_j) > j_0 - 1 + B_1 q^i.
\]
If $j > j_0$, then this follows immediately from the definition of $B_1$.  This leaves only the finitely remaining values $j < j_0$.  We must show that for $i$ sufficiently large, we have
\[
\left(\ord_{\pi_0} (ja_j) - B_1\right) q^i > j_0 - j,
\]
and this is clear.  After choosing $B$ to account for the finitely many values of $i$ which are not sufficiently large, we find that for \emph{every} $i \geq 1$, we have
\[
\ord_{\pi_i} \varphi'(\pi_i) \leq Bq^i.
\]

We now show how this implies the result.  By Weierstrass preparation, we can write
\[
\varphi(x) - \pi_{i-1} = w(x) h(x),
\]
where $h(x)$ is a monic polynomial of degree $q$ and where $w(x)$ is a unit in $\O_{K(\pi_{i-1})}[\![x]\!]$.  
Taking the derivative of both sides and evaluating at $\pi_i$ yields
\[
\varphi'(\pi_i) = w'(\pi_i) h(\pi_i) + w(\pi_i) h'(\pi_i) = w(\pi_i) \prod_{\sigma \neq \text{id}} (\pi_i - \sigma(\pi_i)).
\]
(Here the product is over the $q-1$ embeddings of $K(\pi_i)$ into $\overline{L}$ which fix $K(\pi_{i-1})$.)  Hence we have
\[
\ord_{\pi_i} (\pi_i - \sigma(\pi_i)) \leq \ord_{\pi_i} (\varphi'(\pi_i)) \leq Bq^i,
\]
as required.
\end{proof}

\begin{lemma} \label{doesn't fix subfield}
Let $\tau$ denote a non-trivial embedding of $K(\pi_i)$ into $\overline{L}$, fixing $K$.  Then
\[
\ord_{\pi_i}(\tau(\pi_i) - \pi_i) \leq Bq^i,
\]
where $B$ is the same as in Lemma~$\ref{lower numbering lemma}$.
\end{lemma}

\begin{proof}
We already know the result if $\tau$ fixes $K(\pi_{i-1})$, so choose a $\tau$ which doesn't fix $K(\pi_{i-1})$.
Assume towards contradiction that 
\[
\ord_{\pi_i}(\tau(\pi_i) - \pi_i) > Bq^i.
\]
Considering the power series $\varphi$ which relates $\pi_i$ and $\pi_{i-1}$, we then have
\[
\ord_{\pi_i}(\tau(\pi_{i-1}) - \pi_{i-1}) \geq \ord_{\pi_i}(\tau(\pi_{i}) - \pi_{i}) > Bq^i,
\]
and so
\[
\ord_{\pi_{i-1}}(\tau(\pi_{i-1}) - \pi_{i-1}) > Bq^{i-1}.
\]
Using descent on $i$, we reach a contradiction.
\end{proof}

\begin{lemma} \label{galois closure lemma}
Let $\overline{K(\pi_i)}$ denote the Galois closure of $K(\pi_i)$ over $K$.  Let $\varpi_i$ denote a uniformizer of $\overline{K(\pi_i)}$.  
There exists a constant $\overline{B}$, independent of $i$, such that if $\tau \in \Gal(\overline{K(\pi_i)}/K)$ is nontrivial and acts trivially on the maximal unramified subextension.  Then 
\[
\ord_{\varpi_i}(\tau(\varpi_i) - \varpi_i) \leq \overline{B}q^i.
\]
\end{lemma}

\begin{proof}
We have
\begin{align*}
\ord_{\varpi_i}(\tau(\varpi_i) - \varpi_i) &= \ord_{\varpi_i}(\pi_i) \cdot \ord_{\pi_i}(\tau(\varpi_i) - \varpi_i) \\
&\leq \ord_{\varpi_i}(\pi_i) \cdot \ord_{\pi_i}(\tau(\pi_i) - \pi_i) \\
\intertext{(to obtain this bound, write $\pi_i$ as a polynomial in $\varpi_i$ with coefficients in the maximal unramified subextension)}
&\leq \ord_{\varpi_i}(\pi_i) \cdot Bq^i,
\end{align*}
where the last inequality follows from Lemma~\ref{doesn't fix subfield}.
Because the degree of $\overline{K(\pi_i)}/K(\pi_i)$ is bounded by $(q-1)!$ (independent of $i$), taking $\overline{B} = (q-1)!B$ proves the inequality.
\end{proof}

\newcommand{\Fix}{\textup{Fix}}

\begin{lemma} \label{psi is identity}
For any $i$ and any $x \leq Aq^i$, we have
\[
\psi_{K(\pi_i)/K(\pi_{i-1})}(x) = x.
\]
In particular, for any $i$ and any $x \leq Aq^{i+1}$
\[
\psi_{L/K(\pi_i)}(x) = x.
\]
\end{lemma}

\begin{proof}
For the first assertion, it suffices to show that $\phi_{K(\pi_i)/K(\pi_{i-1})}(x) = x$ for $x \leq Aq^i$.  By Lemma~\ref{lower numbering lemma}, every embedding $\sigma$ of $K(\pi_i)$ into $\overline{L}$ which fixes $K(\pi_{i-1})$ satisfies
\[
\ord_{\pi_i} ( \sigma(\pi_i) - \pi_i) \geq Aq^i.
\]  
Thus using the definition of \cite[\S{1.1.1}]{Wintenberger}, we see that $\phi_{K(\pi_i)/K(\pi_{i-1})}(x) = x$ for $x \leq Aq^i$.  

For the second assertion, we use the first assertion and the fact that for any $j \geq 1$
\[
\psi_{L/K(\pi_i)} = \psi_{L/K(\pi_{i+j})} \circ \psi_{K(\pi_{i+j})/K(\pi_{i+j-1})} \circ \cdots \circ \psi_{K(\pi_{i+1})/K(\pi_i)}.
\]
\end{proof}

\begin{lemma} \label{integer sequence for psi}
There exists an increasing and unbounded sequence of real numbers (independent of $i,j$) $A_0, A_1, A_2, \ldots$ such that for all $i,j$ we have
\[
\psi'_{K(\pi_{i+j})/K(\pi_i)}(x) \leq \left\{ 
\begin{array}{ccrl} 
1 & \text{ for } &&x \leq A_0 q^{i+1} \\
q & \text{ for } &A_0 q^{i+1} \leq &x \leq A_1 q^{i+1} \\
\cdots \\
q^j & \text{ for } &A_{j-1}q^{i+1} \leq &x < \infty.
\end{array} \right.
\]
\end{lemma}

\begin{proof}
Note first that for any finite extension of fields $E/F$, we have that $\psi'_{E/F}(x) \leq [E:F]$.  This fact will be used without further comment below. 

Let $A$ be defined as above.  We claim that we may take $A_{j-1} = jA - \frac{(j-1)A}{q}$.  (Because $A_{j-1} \geq \frac{j}{2} A$, these numbers do approach infinity.) 
We prove this using induction on $j$.  The claim for $j = 1$ follows from Lemma~\ref{psi is identity}, together with the fact that $[K(\pi_{i+1}):K(\pi_i)] = q$.  Now assume the claim for some fixed $j-1$. Our strategy is to use the fact that 
\[
\psi'_{K(\pi_{i+j})/K(\pi_i)}(x) = \psi'_{K(\pi_{i+j})/K(\pi_{i+j-1})} \Bigg( \psi_{K(\pi_{i+j-1})/K(\pi_i)}(x) \Bigg) \cdot \psi'_{K(\pi_{i+j-1})/K(\pi_i)}(x),
\]
together with Lemma~\ref{psi is identity} and the inductive hypothesis.  Because 
\[
\psi'_{K(\pi_{i+j})/K(\pi_{i+j-1})}(y) \leq 1 
\]
for $y \leq Aq^{i+j}$ and
\[
\psi'_{K(\pi_{i+j})/K(\pi_{i+j-1})}(y) \leq q
\]
for all $y$, we are finished by induction if we can show that 
\[
\psi_{K(\pi_{i+j-1})/K(\pi_i)}(A_{j-1}q^{i+1}) \leq Aq^{i+j}.
\]
Using our slope bounds and $\psi(0) =  0$, we know that for all $x \geq A_{j-2}q^{i+1}$, we have 
\[
\psi_{K(\pi_{i+j-1})/K(\pi_i)}(x)  \leq A_0 q^{i+1} + q(A_1 - A_0)q^{i+1} + \cdots + q^{j-1} (x - A_{j-2}q^{i+1}).
\]
Evaluating at $x = A_{j-1}q^{i+1}$ and using our formula for $A_0, \ldots, A_{j-1}$ in terms of $A$, we complete the proof. 
\end{proof}

\begin{corollary} \label{psi corollary}
Let $\overline{B}$ denote some fixed constant.  There exists a constant $D$, depending on $\overline{B}$ but independent of $i$, such that
\[
Dq^i \geq \psi_{L/K(\pi_i)}(\overline{B}q^i).
\]
\end{corollary}

\begin{proof}
In the notation of Lemma~\ref{integer sequence for psi}, choose a $j$ such that $\overline{B}q^i \leq A_j q^{i+1}$; note that this does not depend on $i$.  Then as a coarse bound, we have $\psi_{L/K(\pi_i)}(\overline{B}q^i) \leq q^j \overline{B}q^i$.  Hence for $D$ we may take $q^j \overline{B}$.  Note that this is independent of $i$.
\end{proof}

\begin{lemma} \label{galois tower ramification groups}
Let $K \subseteq K_1 \subseteq K_2$ be two Galois extensions of $K$, with Galois groups $G_1, G_2$.  Then
\[
\Fix(G_2^u) \cap K_1 = \Fix(G_1^u). 
\]    
\end{lemma}

\begin{proof}
Let $H_1 \trianglelefteq G_2$ denote the subgroup corresponding to the field $K_1$; we thus have $G_2/H_1 \cong G_1$.  By properties of ramification groups in the upper numbering, we have
\[
(G_2^u H_1/H_1) = (G_2/H_1)^u \cong G_1^u.
\] 
The claim follows immediately.
\end{proof}

\begin{proposition} \label{pi in fixed fields}
There exists a constant $D$, independent of $i$, such that $K(\pi_i)$ is contained in the fixed field corresponding to $G_K^{\phi_{L/K}(Dq^i)} G_L$.
\end{proposition}

\begin{proof}
Write $\overline{K(\pi_{i+j})}$ for the Galois closure of $K(\pi_{i+j})$ over $K$.  
It suffices to show that there exists a constant $D$, independent of $i$ and $j$, such that
\[
K(\pi_i) \subseteq \Fix\left(\Gal(\overline{K(\pi_{i+j})}/K)_{\psi_{\overline{K(\pi_{i+j})}/K} \circ \phi_{L/K}(Dq^i)}\right).
\]
By Lemma~\ref{galois tower ramification groups}, it suffices to consider the case $j = 0$.  We want to show that if $\sigma(\pi_i) \neq \pi_i$, then 
\[
\sigma \not\in \Gal(\overline{K(\pi_{i})}/K)_{\psi_{\overline{K(\pi_{i})}/K} \circ \phi_{L/K}(Dq^i)} = \Gal(\overline{K(\pi_{i})}/K)_{\psi_{\overline{K(\pi_{i})}/K(\pi_i)} \circ \phi_{L/K(\pi_i)}(Dq^i)}
\]
Because $\psi(x) \geq x$, it suffices to prove that there exists a $D$, independent of $i$, such that 
\[
\sigma \not\in \Gal(\overline{K(\pi_{i})}/K)_{ \phi_{L/K(\pi_i)}(Dq^i)}.
\]
By Lemma~\ref{galois closure lemma}, we know 
\[
\sigma \not\in \Gal(\overline{K(\pi_{i})}/K)_{ \overline{B}q^i}.
\]
Because $\psi_{L/K(\pi_i)}$ is an increasing function, it suffices to show that there exists a constant $D$, independent of $i$, such that
\[
Dq^i \geq \psi_{L/K(\pi_i)}(\overline{B}q^i).
\]
We are finished by Corollary~\ref{psi corollary}.
\end{proof}

\begin{lemma} \label{galois index bound}
There exists a constant $N$, independent of $i$, such that 
\[
[G_K : G_K^{\phi_{L/K}(Dq^i)}G_L] \leq Nq^i
\]
for all $i$.
\end{lemma}

\begin{proof}
By (\ref{hypothesis elementary}), using the fact that $\psi_{L/K}$ is the inverse to $\phi_{L/K}$, we have 
\[
[G_K : G_K^{\phi_{L/K}(Dq^i)}G_L] \leq \frac{Dq^i}{c}.
\]
Thus we may take $N = D/c$.
\end{proof}

Let $D$ be as in Proposition~\ref{pi in fixed fields}.  Let $v_i := \phi_{L/K}(Dq^i)$ and let $F_i$ denote the fixed field corresponding to $G_K^{v_i} G_L$. 

\begin{lemma} \label{no infinite sequence of indices}
There does not exist an infinite sequence $0 = i_0 < i_1 < i_2 < \cdots$ such that 
\[
[F_{i_{j+1}}: F_{i_j}] > q^{i_{j+1} - i_j}
\] 
for all $j$.  
\end{lemma}

\begin{proof}
 By Proposition~\ref{pi in fixed fields}, we know that $K(\pi_i) \subseteq F_i$.  We also know that $[K(\pi_i):K] = q^i$ for all $i$.  By these observations and Lemma~\ref{galois index bound}, we have that
\[
q^i \mid [F_i:K] \qquad \text{and} \qquad [F_i:K] \leq Nq^i
\]
for all $i$.  

Now assume towards contradiction that there does exist a sequence as in the statement of the lemma.  We have $[F_{i_1} : F_0] > q^{i_1}$, and also $q^{i_1} \mid [F_{i_1} : F_0]$, so $[F_{i_1} : F_0] \geq 2q^{i_1}$.  We also have $[F_{i_2} : F_{i_1}] > q^{i_2 - i_1}$, so $[F_{i_2} : F_0] > 2q^{i_2}$, and also $q^{i_2} \mid [F_{i_2} : F_0]$, so $[F_{i_2} : F_0] \geq 3q^{i_2}$.  Continuing in this way, we eventually pass $Nq^i$, and thus reach a contradiction.  
\end{proof}

\begin{corollary} \label{pi and elementary subfields}
There exists an $i_0 >\!> 0$ such that $[F_{i_0+j} : F_{i_0}] \leq q^{j}$ for all $j$.  In particular, there exists a field $K_{i}$ in the tower of elementary extensions such that 
\[
q^m \mid [K_{i+j} : K_i]
\] 
implies $\pi_{i_0 +m} \in K_{i+j}$.  
\end{corollary}

\begin{proof}
The first assertion follows immediately from Lemma~\ref{no infinite sequence of indices}.  We now prove the second assertion.  Notice that each of the fields $F_j$, being the fixed field of $G_K^{v_j}G_L$, occurs in the tower of elementary extensions, since the tower of elementary extensions contains the fixed field of $G_K^v G_L$ for every $v$.  Take $K_i = F_{i_0}$.  We know that $\pi_{i_0 + j} \in F_{i_0 + j}$ by Proposition~\ref{pi in fixed fields}, and we know that $F_{i_0 + j}$ occurs somewhere in the tower of elementary extensions.  Because $[F_{i_0 + m} : K_i] = [F_{i_0 + m} : F_{i_0}] \leq q^m$, we deduce that if $q^m \mid [K_{i+j} : K_i]$, then $\pi_{i_0 + m} \in F_{i_0 + m} \subseteq K_{i+j}$, as required.  
\end{proof}

The preceding results provide a connection between the elementary fields $K_n$ and the fields $K(\pi_i)$.  The fields $K_n$ are important because of their appearance in the definition of $\ALK$ (and because they are canonically associated to the extension $L/K$).  The fields $K(\pi_i)$ are important because in each such field we know an explicit uniformizer $\pi_i$, and from these uniformizers we can produce a uniformizer for $\e_{L/K}^+$.  

\begin{lemma} \label{a uniformizer for E}
Let $\mathfrak{b}$ be as in $(\ref{bideal})$ and let $\mathfrak{a} = (\pi, \mathfrak{b})$.  Let $q = p^a$.  A uniformizer for $\e_{L/K}^+$ is given by the $p$-power compatible system in
\[
\varprojlim_{x \mapsto x^p} \O_{L}/\mathfrak{a}
\]
which has $ai$ component equal to $\pi_i \bmod \mathfrak{a}$.
\end{lemma}

\begin{proof}
Using that the system of fields $K(\pi_i)$ is cofinal among the system of all intermediate fields between $L$ and $K$, we may express the valuation ring inside of $X_K(L)$ as
$\varprojlim_{\text{Nm}} \O_{K(\pi_i)}$.  Using Weierstrass preparation, write $\varphi(x) = w(x) h(x)$, where $w(x) \in \O_F[\![x]\!]$ is a unit and where $h(x)$ is a monic degree $q$ polynomial with no constant term.  Because $\varphi(x) \equiv x^q \bmod \pi$, we must have $w(x) \equiv 1 \bmod \pi$.  Because $\varphi(\pi_i) = \pi_{i-1}$, we must have that $h(\pi_i) \equiv \pi_{i-1} \bmod \pi$.  

Let $f_i(x) \equiv x^q \bmod \pi_{i-1} \in \O_{K(\pi_{i-1})}[x]$ denote the minimal polynomial for $\pi_i$ over $K(\pi_{i-1})$.  We compare this polynomial to $h(x) - \pi_{i-1}$.  Write
\[
f_i(x) - (h(x) - \pi_{i-1}) = a_1 x^{q-1} + \cdots + a_q \in \O_{K(\pi_{i-1})}[x].
\]
Evaluating this polynomial at $\pi_i$ yields a sum of $q$ elements of $\O_{K(\pi_i)}$ which have distinct $\pi_{i}$-valuations.  On the other hand, evaluating this polynomial at $\pi_i$ yields an element divisible by $\pi$.  This shows that each of the $q$ terms, and in particular the constant term, must be divisible by $\pi$.  We deduce that the constant term of $f_i(x)$ differs from $-\pi_{i-1}$ by a multiple of $\pi$.  In particular, $\Nm_{K(\pi_i)/K(\pi_{i-1})}(\pi_i) \equiv \pi_{i-1} \bmod \pi$.  (This holds for both $p = 2$ and for $p$ odd.)  Thus, for the ideal $\mathfrak{a}$ as in the statement of the lemma, we have that the elements $(\pi_i \bmod \mathfrak{a})$ are norm-compatible.  As in the proof of Lemma~\ref{sindep}, this sequence lifts to a norm-compatible family of uniformizers, which in turn is a uniformizer for $\e_{L/K}^+$.  
\end{proof}

Recall that the map $\beta$ was defined in Definition~\ref{def: beta} (see also Proposition~\ref{image of beta}).  The following result shows that the image of $\beta$ contains the uniformizer of Lemma~\ref{a uniformizer for E} raised to some power of $q$.

\begin{proposition} \label{beta uniformizer}
For $i > 0$, set $\pi_{-i} := \varphi^i(\pi_0)$.  
There exists some $i_1 \geq 0$ and some $x \in \ALK$ such that $\beta(x) = (\pi_{n-i_1} \bmod \mathfrak{b})_n$.
\end{proposition}

\begin{proof}
We first show that there exists a Frobenius-compatible family of elements $\underline{x}_n \in W_{\pi}(\O_{K(\pi_n)})$ such that their first components satisfy
\[
x_{n0} = \pi_n.
\]
Consider the unique continuous $\O_F$-algebra homomorphism $\varphi: \O_F[\![x]\!] \rightarrow \O_F[\![x]\!]$ which sends $x$ to $\varphi(x)$.  By Proposition~\ref{lambda prop}, there is a unique ring homomorphism
\[
\lambda_{\varphi}:  \O_F[\![x]\!] \rightarrow W_{\pi}(\O_F[\![x]\!])
\]
which is a section to the projection $W_{\pi}(\O_F[\![x]\!]) \rightarrow  \O_F[\![x]\!]$ and which satisfies $\Fr \circ \lambda_{\varphi} = \lambda_{\varphi} \circ \varphi$.  By Corollary~\ref{ghost comps lambda}, the element $\lambda_{\varphi}(x)$ has ghost components $(x, \varphi(x), \varphi^2(x), \ldots)$.  Now for any $n$, consider the $\O_F$-algebra homomorphism $\O_F[\![x]\!] \rightarrow \O_{K(\pi_n)}$ determined by $x \mapsto \pi_n$.  This determines a map $W_{\pi}(\O_F[\![x]\!]) \rightarrow W_{\pi}(\O_{K(\pi_n)})$.  Because the ghost components are defined by polynomials with coefficients in $\O_F$, we see that the composite
\[
\O_F[\![x]\!] \rightarrow W_{\pi}(\O_F[\![x]\!]) \rightarrow W_{\pi}(\O_{K(\pi_n)})
\]
sends $x$ to a Witt vector with ghost components $(\pi_n, \varphi(\pi_n), \varphi^2(\pi_n), \ldots) = (\pi_n, \pi_{n-1}, \ldots, \pi_0, \varphi(\pi_0), \ldots)$.  We define $\underline{x}_n$ to be the image of $x$ under this composite.  In particular, $x_{n0} = \pi_n$, and $\Fr(\underline{x}_n) = \underline{x}_{n-1}$.  This completes our construction of the elements $\underline{x}_n \in W_{\pi}(\O_{K(\pi_n)})$.

To construct the desired element in $\ALK$, we need to work with the elementary subfields $K_m$, rather than the fields $K(\pi_n)$.  Let $K_i$ be is as in Corollary~\ref{pi and elementary subfields}, and let $i_1$ be such that $[K_i : K_1] \mid q^{i_1}$.   Using notation of the previous paragraph and the statement of the proposition, we claim that $(\underline{x}_{n-i_1})_n$, which we know is an element of $\wt{\a}^+$, is in fact an element of $\ALK$.  Recalling the definition of $\ALK$, we need to show that if $q^n \mid [K_m : K_1]$, then $\pi_{n-i_1} \in K_m$.  We deduce this as follows:
\begin{align*}
q^n &\mid [K_m : K_1] \\
q^n &\mid [K_m : K_i] [K_i : K_1] \\
q^{n-i_1} &\mid [K_m : K_i] \\
\intertext{and by Corollary~\ref{pi and elementary subfields}, this implies}
\pi_{i_0 + n - i_1} &\in K_m,
\end{align*}
and in particular $\pi_{n - i_1} \in \O_{K_m}$.  This completes the proof, at least for $n - i_1 \geq 0$, but if $n - i_1 < 0$, then $\pi_{n - i_1} = \varphi^{i_1 - n}(\pi_0) \in \O_K \subseteq \O_{K_m}$ for all $m$.
\end{proof}

\begin{proposition} \label{residue ring}
As subrings of $\wt{\e}^+$, we have $\ALK/\pi\ALK \subseteq \e_{L/K}^+$, and moreover $\e_{L/K}$ is a finite purely inseparable extension of $\Frac \left(\ALK/\pi\ALK\right)$. 
\end{proposition}

\begin{proof}
The inclusion $\ALK/\pi\ALK \subseteq \e_{L/K}^+$ follows from Proposition~\ref{image of beta}.  We know that $\e_{L/K}^+$ is abstractly isomorphic to $k[\![u]\!]$.  We also know that the image of $\beta$ contains the coefficient ring $k$ by Proposition~\ref{prop: W(k)-algebra etc}.  By Proposition~\ref{beta uniformizer}, we know that there exists $v \in \ALK$ such that $\beta(v) = u^{q^{i_1}}$.  Because $\ALK \subseteq \wt{\a}^+$ is weakly closed (Proposition~\ref{ALK is weakly closed}), we deduce that the image of $\beta$ contains $k[\![u^{q^{i_1}}]\!]$, which completes the proof. 
\end{proof}

Proposition~\ref{residue ring} immediately implies our second main theorem from the introduction, Theorem~\ref{MT-iterate}.  

%These results combine to prove the main theorem of the paper.

%\begin{theorem} \label{theorem for iterated extensions}
%Let $F$ denote a finite extension of $\Q_p$ with residue field $k_F$ of cardinality $q$ and with uniformizer $\pi$.  Let $K/F$ denote a finite extension and let $L/K$ denote a $\varphi$-iterate extension, as in Definition~\ref{iterated def}.  The ring $\a_{L/K}^+$ defined in Definition~\ref{new a def} is a canonical Cohen ring for \chris{this should allow for a purely inseparable extension} $\e_{L/K}^+$, i.e., it satisfies the properties (\ref{top})--(\ref{red}) enumerated in the introduction.
%\end{theorem} 

%\begin{proof}
%It was shown in Proposition~\ref{prop: W(k)-algebra etc} that $\a_{L/K}^+$ is $\pi$-torsion free, $\pi$-adically complete and separated.  It carries an action of Frobenius and if $L/K$ is Galois, it carries an action of the Galois group, by \chris{insert reference}.  It is weakly closed in $\wt{\a}^+ \otimes_{W(k_F)} \O_F$ by Proposition~\ref{ALK is weakly closed}.  Lastly, we have $\a_{L/K}^+/\pi\a_{L/K}^+ \cong \e_{L/K}^+$ by \chris{this is not quite true, purely inseparable extension} Proposition~\ref{residue ring}. 
%\end{proof}

\begin{remark}
Continue to assume $L/K$ is a $\varphi$-iterate extension.  
Here we provide a significantly simpler description of a subring of $W_{\pi}(\wt{\e}_{L/K}^+) \cong \varprojlim W_{\pi}(\O_{\wh{L}})$ which lifts $\e_{L/K}^+$.  Consider the continuous map
\[
\iota_{\underline{\pi}} : {W_{\pi}(k)}[\![u]\!] \rightarrow \varprojlim W_{\pi}(\O_{\wh{L}})
\]
which sends $u$ to $(\underline{x}_n)_n$, where $\underline{x}_n$ is the (unique) Witt vector in $W_{\pi}(\O_{\wh{L}})$ with ghost components $(\pi_n, \pi_{n-1}, \ldots, \pi_0, \varphi(\pi_0), \ldots)$.  (Such Witt vectors exist by the Dwork Lemma, Proposition~\ref{dwork}.)  We write $\a_{\underline{\pi}}$ for the image of $\iota_{\underline{\pi}}$.  It is a $\pi$-Cohen ring for $\e_{L/K}^+$, and it is related to our ring $\a_{L/K}^+$ by $\varphi^{i_1}(\a_{\underline{\pi}}) \subseteq \ALK,$ where $i_1$ is as in Proposition~\ref{beta uniformizer}.
The downside to $\a_{\underline{\pi}}$ is that it seems to depend heavily on the choice of the elements $\pi_i$, and hence it is not obvious from this point-of-view that $\ALK$ can be canonically associated to the $\varphi$-iterate extension $L/K$.  In particular, if $L/K$ is Galois, it is not clear that $\a_{\underline{\pi}}$ is stable under the action of $G_K$ on $\varprojlim W_{\pi}(\O_{\wh{L}})$.  Also it is not clear how to define an analogue of $\a_{\underline{\pi}}$ for an arbitrary strictly APF extension $L/K$ (i.e., for extensions other than $\varphi$-iterate extensions).
\end{remark}

\begin{remark}
%Assume now that $L/K$ is a \emph{Galois} $\varphi$-iterate extension.  
Our Proposition~\ref{residue ring}, in conjunction with Proposition~\ref{abelian aut group proposition},
%implies that $\ALK$ is isomorphic to $W_{\pi}(k)[\![u]\!]$ and is stable under actions of Frobenius and $G_K$.  In conjunction with Theorem~5.1 of \cite{BergerLift}, this 
implies that if $L/K$ is both a $\varphi$-iterate extension and Galois, then its Galois group is necessarily abelian.
	It would be interesting to prove this fact directly.
\end{remark}

\bibliography{canonical}
\bibliographystyle{plain}

\end{document}